[1994/12/01]

\documentclass{amsart}
\usepackage{amssymb}
\usepackage{amsmath,amssymb,amsfonts,amsthm,graphics,
latexsym, amscd, amsfonts, epsfig, eepic,epic}
\usepackage{mathrsfs}
\usepackage{algorithm}
\usepackage{algorithmic}
\usepackage{color}

\usepackage{eucal}
\usepackage{eufrak}
\usepackage[all]{xypic}
\usepackage{xspace}

\newtheorem{thm}{Theorem}[section]
\newtheorem{lem}[thm]{Lemma}
\newtheorem{prop}[thm]{Proposition}
\newtheorem{cor}[thm]{Corollary}

\theoremstyle{definition}

\theoremstyle{remark}
\newtheorem{rem}{Remark}[section]

\makeatletter \@addtoreset{equation}{section}

\newcommand{\thmref}[1]{Theorem~\ref{#1}}

\newcommand{\lemref}[1]{Lemma~\ref{#1}}
\newcommand{\corref}[1]{Corollary~\ref{#1}}

\newcommand{\G}{{G(X|Y)}}

\def\a{\alpha}
\def\b{\beta}

\def\g{\gamma}

\def\l{\lambda}

\def\p{\partial}
\def\vphi{\varphi}

\def\L{\Lambda}
\def\G{\Gamma}

\def\g{\mathfrak g}

\def\and{\quad{\rm and}\quad}

\def\eps{\epsilon}

\let\lra=\longrightarrow

\def\mapright{\xrightarrow}
\def\mapright\#1{\,\smash{\mathop{\lra}\limits^{\#1}}\,}

\def\om{\omega}

\def\tri{\triangle}
\def\ul{\underline}
\newcommand{\gij}{\ensuremath{g_{i\bar j}}}
\newcommand{\as}{\underline{S}}

\numberwithin{equation}{section}

\DeclareMathOperator{\osc}{Osc}

\makeatletter

\newcommand{\Rmnum}[1]{\expandafter\@slowromancap\romannumeral #1@}

\title{Stability of K\"ahler-Ricci flow in the space of K\"ahler metrics}
\author[Zheng]{Kai Zheng}
  \address{Academy of Mathematics and Systems Sciences, Chinese Academy of Sciences, Beijing, 100190, P.R.~China.}
  \email{kaizheng@amss.ac.cn}
\date{}
\begin{document}
\maketitle

\begin{abstract}
In this paper, we prove that on a Fano manifold $M$ which
admits a K\"ahler-Ricci soliton $(\om,X)$,
if the initial K\"ahler metric $\om_{\vphi_0}$ is close to $\om$ in some weak sense,
then the weak K\"ahler-Ricci flow exists globally and converges in Cheeger-Gromov sense.
Moreover, if $\vphi_0$ is also $K_X$-invariant,
then the weak modified K\"ahler-Ricci flow
converges exponentially
to a unique K\"ahler-Ricci soliton nearby.
Especially,
if the Futaki invariant vanishes, we may delete the $K_X$-invariant assumption.
The methods based on the metric geometry of the space of the K\"ahler metrics are potentially applicable to other stability problem of geometric flow near a critical metric.
\end{abstract}
\section{Introduction}
Ricci flow, introduced by Hamilton \cite{MR664497},
plays an important role in understanding the geometric and topological
structure of the manifolds which it lives on.
We call the Ricci flow a K\"ahler-Ricci flow, if the underlying manifold is a K\"ahler manifold.
Furthermore, the normalized K\"ahler-Ricci flow is given by \begin{equation}\label{KRF}
\begin{cases}
\frac{\p}{\p t}\om&=-Ric+\lambda\om,\\
\om(0)&=\om_{\varphi_0}
\end{cases}
\end{equation}
in which $\om(0)$ stays in the canonical class $2\pi C_1(M)$ and $\l$ is the sign of the first Chern class.
Cao \cite{MR799272} first showed that K\"ahler-Ricci flow \eqref{KRF} has long time existence
and converges to a K\"ahler-Einstein metric
when the first Chern class is negative or zero.
Now we restrict ourselves in the situation that the first Chern class is positive.
Since the K\"ahler-Ricci flow preserves the K\"ahler class,
we rewrite
the K\"ahler-Ricci flow in the potential level as
\begin{equation}\label{KRF potential}
\begin{cases}
\frac{\p\vphi}{\p t}&=\log\frac{\om^n_\vphi}{\om^n}
+\vphi-h_\om+a(t),\\
\vphi(0)&=\vphi_0
\end{cases}
\end{equation}
where $a(t)$ is a constant depending on $t$
and $h_\om$ is the Ricci potential of the reference metric $\om$
defined by
\begin{eqnarray}\label{ricci pot of back}
\sqrt{-1}\p\bar\p h_\om=Ric(\om)-\om \text{ and }
\int_Me^{h_\om}\om^n=Vol(M).
\end{eqnarray}
In Perelman \cite{Perelman}, he obtained an estimate of the K\"ahler-Ricci flow
and proved that
the K\"ahler-Ricci flow converges to a K\"ahler-Einstein metric
in the sense of Cheeger-Gromov when one exists for any initial K\"ahler metric.
Later on,
Sesum-Tian \cite{MR2427424} gave the detailed proof.
Furthermore, Tian-Zhu \cite{MR2291916} extended it to the case of K\"ahler-Ricci soliton for a $K_X$-invariant initial metric. In which, a K\"ahler-Ricci soliton is a K\"ahler metric such that
if there is a holomorphic vector field $X$ such that
\begin{align}\label{sol}
L_X\om=Ric-\om.
\end{align}
Since the right side of the equation \eqref{sol} is real-valued, we obtain $L_{\Im X}\om=0$ and $\Im X$, the imaginary part of $X$, generates a one-parameter isometry group $K_X$.

In order to study the asymptotic behavior of the K\"ahler-Ricci flow, we consider the stability problem.
I.e. on a K\"ahler manifold $M$ admits a K\"ahler-Ricci soliton $(\om,X)$,
in what kind of neighborhood of $\om$, the K\"ahler-Ricci flow
with its initial datum stays, converges
in some sense (maybe exponentially) to a K\"ahler-Ricci soliton.

This stability problem of the K\"ahler-Ricci flow has been initiated and studied by many people, for complete references we refer to Chen-Li \cite{ChenLi-2009}. In
Chen-Li \cite{ChenLi-2009}
and Tian-Zhu \cite{zhu-2008}, they consider perturbing both the initial metric and the complex structure near a K\"ahler-Einstein metric.

In this paper, we focus on perturbing the initial metric near the K\"ahler-Ricci soliton without changing the complex structure.
Firstly, we give a direct proof of the long time existence and convergence in Cheeger-Gromov sense within the frame of Donaldon's programme \cite{MR2103718}. The proof of which is based on the geometry of the space of K\"ahler metrics. Next, we derive the exponential convergence and the uniqueness of the limit via calculating the energy function.
Set $\mathcal{N}(\eps_0;B,p)$ be a small neighborhood of the zero function depends on $\eps_0$, $B$ and $p$ which will be specified in Section \ref{WF}.
The main results of this paper are given as follows:
\begin{thm}\label{sta sol}
On a manifold
admits a K\"ahler-Ricci soliton $(\om,X)$, there exists a positive constant $\eps_0$,
if the initial potential $\vphi_0$ stays in $\mathcal{N}(\eps_0;B,p)$,
then the weak K\"ahler-Ricci flow exists globally and converges in Cheeger-Gromov sense.
Moreover, if $\vphi_0$ is $K_X$-invariant,
the weak modified K\"ahler-Ricci flow
converges exponentially to a unique K\"ahler-Ricci soliton nearby.
\end{thm}
When the Futaki invariant vanishes, it is obvious that the holomorphic vector fields $X=0$ and the K\"ahler-Ricci soliton is a K\"ahler-Einstein metric, then we have
\begin{thm}\label{sta KE}
On a K\"ahler-Einstein manifold, there exists a positive constant $\eps_0$,
if the initial potential $\vphi_0$ stays in $\mathcal{N}(\eps_0;B,p)$,
then the weak K\"ahler-Ricci flow exists globally and
converges exponentially to a unique K\"ahler-Einstein metric nearby.
\end{thm}
Simon \cite{MR727703} studied the asymptotic behavior of the gradient flow
of the variation problem by so called the Lojasiewicz-Simon inequality
which compares the distance to the critical set
with the norm of the gradient of the functional
in the $L^2$ space under the condition
that the functional should be analytic. The underlying idea is to reduce the infinite-dimensional problem to a finite-dimensional problem. Perelman \cite{Perelman1} introduced a new functional called $\mu$ functional and pointed out that the Ricci flow is the gradient flow of the $\mu$ functional up to a diffeomorphism.

However, in this paper, we do not apply the Lojasiewicz-Simon inequality to the $\mu$ functional directly. In fact, we provide a new approach to study the asymptotic behavior of the flow which is even only a pseudo-gradient flow of some functional,
since
in K\"ahler setting, geometry gives us more information.
To be precise, the critical set in the space of K\"ahler metrics
is a finite dimension Riemannian symmetric space we will explained later.

Since the K\"ahler-Ricci flow is the pseudo-gradient flow of the $K$-energy,
in order to make the mechanism of our proof more clear, we firstly prove \thmref{sta KE} under the assumption that the $C^{2,\a}$ norm of $\vphi_0$ is small. Furthermore, we generalize our approach to the case of K\"ahler-Ricci soliton, \thmref{sta sol}.

We sketch our proof of \thmref{sta sol} and \thmref{sta KE} as follows.  We first prove the K\"ahler-Ricci flow \eqref{KRF potential} after pulling back by the corresponding holomorphic transformations will always stay in a small neighborhood near the background K\"ahler-Einstein metric. When $M$ has no nontrivial holomorphic vector fields, it is not necessary to find the transformations and in Section \ref{Nnhvf} the proof of which is given. However,
in general,
when $M$ admits nontrivial holomorphic vector fields, in Section \ref{CHT} we develop a new method to pick up the appropriate transformations following the trace of the K\"ahler-Ricci flow in the space of normalized K\"ahler potential $\mathcal{H}_0$ (c.f. \eqref{SKM: H0}). It has been shown by
Mabuchi \cite{MR909015}, Donaldson \cite{MR1736211} and  Semmes \cite{MR1165352} independently that $\mathcal{H}_0$ is a infinite dimensional symmetry space of negative curvature. Later, Chen \cite{MR1863016} proved $\mathcal{H}_0$ is also a metric space.
Since the space of potentials of K\"ahler-Einstein metrics, $\mathcal{E}_0$, is a totally geodesic submanifold
in $\mathcal{H}_0$,
the projection $\rho$ minimizing the distance function from the K\"ahler-Ricci flow to $\mathcal{E}_0$ is uniquely determined.
The Bando-Mabuchi's uniqueness theorem of K\"ahler-Einstein metric \cite{MR946233} implies $\om_\rho$
is different from the reference K\"ahler-Einstein metric
by a holomorphic transformation.
The projection K\"ahler-Einstein metric is exactly the new reference metric we acquired.

Another way to derive a holomorphic transformation (in the Appendix) of $\vphi\in \mathcal{H}_0$ is to
minimize the $I-J$ functional in $\mathcal{E}_0$, which has been introduced by Bando-Mabuchi \cite{MR946233} to prove the uniqueness of the K\"ahler-Einstein metric.
However, their method can not be applied in our case directly, since in general the hessian of $I-J$ functional is not strictly positive, i.e.~ the minimizer is not unique.
Nevertheless, as we observed
when the $C^{2,\a}$ norm of $\vphi$ is small,
the hessian of $I-J$ functional is indeed strictly positive.
Therefore,
the holomorphic transformation
is uniquely determined.

Next, in Section \ref{EC}, we derive the exponential convergence of the K\"ahler-Ricci flow by computing the energy functions and using the Futaki invariant. The key idea is since the geometric quantities such as the Sobolev constant and the Poincare constant are invariant under the holomorphic transformation, the De Giorgi-Nash-Moser iteration can be applied to control $\frac{\p\vphi}{\p t}$.

Then in Section \ref{KRS}, we prove a stability theorem of K\"ahler-Ricci flow near a K\"ahler-Ricci soliton $(\om,X)$ similarly to the case of K\"ahler-Einstein metric.
We first prove the K\"ahler-Ricci flow \eqref{KRF potential} modulo automorphisms will always stay in a small neighborhood near the background K\"ahler-Ricci soliton for arbitrary initial K\"ahler potential with small $C^{2,\a}$ norm.
The key idea is to use Perelman's $\mu$ functional \cite{Perelman1} instead of the $K$-energy, since the hessian of the $\mu$ functional is nonnegative at a K\"ahler-Ricci soliton within the canonical class \cite{zhu-2008}.
Furthermore, we reparametrize the K\"ahler-Ricci flow \eqref{KRF} by the automorphisms $\varsigma(t)$ generated by the real part $\Re X$ of $X$ such that
\begin{equation}\label{mkrf}
\begin{cases}
\frac{\p}{\p t}\om_\phi
&=-Ric(\om_\phi)+\om_\phi+L_{\Re X}\om_\phi,\\
\om_{\phi(0)}&=\om_{\vphi_0}.
\end{cases}
\end{equation}
It is obvious that the K\"ahler-Ricci soliton is the stationary solution of the modified K\"ahler-Ricci flow \eqref{mkrf}. Since the K\"ahler-Ricci soliton $(\om,X)$ is $K_X$-invariant and the K\"ahler-Ricci flow is also invariant under the holomorphic differmorphism, without lose of generality, we assume the initial datum is $K_X$-invariant. Then we generalize the exponential convergence of the K\"ahler-Ricci flow derived in Section \ref{EC} to the modified K\"ahler-Ricci flow \eqref{mkrf}.

Finally, in Section \ref{WF}, at a fixed time, we show that the $C^{2,\a}$ norm of the
potential is small when the initial value is small under
certain weak condition. The main idea is to use the estimate
introduced in \cite{ctz-2008}.

As a corollary of \thmref{sta sol}, we deduce that the limit metric
of the K\"ahler-Ricci flow is unique. Set $\{\vphi(t_i)\}$ be a
sequence of the solution of the K\"ahler-Ricci flow which converges
to a K\"ahler-Einstein metric or K\"ahler-Ricci soliton $g_\infty$,
if there exists, then there exists some $\vphi\in\{\vphi(t_i)\}$
such to the stability-condition given in \thmref{sta sol}. According
to the stability \thmref{sta sol}, the K\"ahler-Ricci flow with
initial-value $\vphi$ converges exponentially to a K\"ahler-Einstein
metric $g^1_\infty$ (or K\"ahler-Ricci soliton respectively).
Furthermore, since we assume that $\{\vphi(t_i)\}\rightarrow
g_\infty$, so $g^1_\infty$ must coincide with $g_\infty$.

We emphasize that our approach using to prove \thmref{sta sol} is also
applicable to the case for the general pseudo-gradient flow.
I.e.~neither the condition ``flow is a gradient flow of some
functional", the Perelman's deep estimate \cite{Perelman}, nor a prior
long time existence of the flow is required. It is possible that our
method can be utilized to solve similar problem of other
geometric flow problems. For instance, to prove the stability theorem of the
pseudo-Calabi flow near a constant scalar curvature K\"ahler (cscK)
metric in \cite{zheng-2009} and of the Calabi
flow near a extremal metric in \cite{zheng-2009cf}.

The paper is organized as follows: In section \ref{Bg} we review the known results of the space of K\"ahler metrics and the well-posedness of the pseudo-Calabi flow (c.f. \eqref{PCF}) we obtain in \cite{zheng-2009}. In Section \ref{Nnhvf}, Section \ref{NHVF} and Section \ref{EC} we first prove theorem \thmref{sta KE} under the assumption that the $C^{2,\a}$ norm of the initial K\"ahler potential is small. Then we prove \thmref{sta sol} under the same assumption in Section \ref{KRS}. Finally, in Section \ref{WF} we explain how to weaken the initial condition to which stated in both \thmref{sta KE} and \thmref{sta sol}. In the Section \ref{ACHT}, we explain another method to choose the holomorphic transformation.

\

\noindent {\bf Acknowledgements:}
The author is grateful to thank Prof. Xiuxiong Chen
who brought him into K\"ahler geometry
and introduced him to this problem.
He is also grateful for Prof. Weiyue Ding
for his constant encouragement and support.
He also wants to express his thanks to
Prof. Xiaohua Zhu for his interest in this problem
and his many helpful discussions.

\section{Notations and basic results}\label{Bg}
Let $M$ be a compact K\"ahler manifold of complex dimension $n$ with positive first Chern class $C_1(M)$ and $\om$ be a K\"ahler form which represents the canonical class $2\pi C_1(M)$.
In a local holomorphic coordinate $z_1,z_2,\dotsm z_n$, $\om$ is expressed by
\[
 \om = \sqrt{-1}\sum\limits_{i=1}^{n} \gij dz^i \wedge
dz^{\bar j}.
\]
The corresponding Rimannian metric is given by
$$ g = \sum\limits_{i=1}^{n} \gij dz^i \otimes dz^{\bar j }.$$
For a K\"ahler metric $\om$, the volume form is
\begin{equation*}
dV=\omega^{n}
={(\sqrt{-1})}^n\det(\gij)
dz^1\wedge dz^{\bar{1}} \wedge\dotsm\wedge dz^n\wedge dz^{\bar{n}}.
\end{equation*}
The Ricci form taking the form
$$Ric
= \sqrt{-1}\sum\limits_{i=1}^{n} R_{i\bar j} dz^i \wedge dz^{\bar{j}}
=- \sqrt{-1}\p \bar \p \log \det \om^n$$
is a closed real $(1,1)$-form and stays in $2\pi C_1(M)$.
According to which, we obtain the scalar curvature satisfies
$$S\om^{n}=n Ric \wedge \om^{n-1}.$$
Furthermore,
a direct calculation gives the average of the scalar curvature
\begin{equation*}
\as =\frac{1}{V}\int_{M} S dV
=\frac{n}{V} \int_{M} Ric \wedge \omega^{n-1}
=n.
\end{equation*}

Let $\mathcal{K}$ be the set of all K\"ahler forms on $M$ representing $2\pi C_1(M)$ and $\mathcal{E}$ be the set of all K\"ahler-Einstein metrics in $\mathcal{K}$. According to $\p\bar\p$ lemma, for any K\"ahler metric $\om'$ in $\mathcal{K}$ there exists a smooth real-valued function $\vphi$ such that $\om'=\om+\sqrt{-1}\p\bar\p\vphi$.
Then the space of K\"ahler potentials of $\mathcal{K}$
is given by
\[
\mathcal{H}
=\{\vphi\in C^{\infty}(M,\mathbb{R})\vert\om+\sqrt{-1}\p\bar\p\vphi\in\mathcal{K}\}.
\]
Apparently, we have a isomorphism $T\mathcal{H}\cong \mathcal{H}\times C^\infty(M,\mathbb{R})$.
Mabuchi \cite{MR909015}, Donaldson \cite{MR1736211}
and Semmes \cite{MR1165352} independently
defined a Riemannian metric on $\mathcal{H}$ by
\[
\int_M f_1 f_2 \om^n_\vphi
\]
for any $f_1,f_2\in T_\vphi \mathcal{H}$.
For any path $\vphi(t)(0\leq t\leq 1)$ in $\mathcal{H}$, the length is given by
\begin{align}\label{SKM: length}
L(\vphi(t))=\int_0^1\sqrt{\int_M\vphi'(t)^2\om^n_{\vphi(t)}}dt
\end{align} and the geodesic equation is
\begin{align*}
\vphi''(t)-\frac{1}{2}|\nabla_t\vphi'(t)|_{\vphi(t)}^2=0
\end{align*}
in which we use $'$ to denote the differentiation in $t$ and $\nabla_t$ to denote the covariant derivative for the metric $g_{\vphi(t)}$.
The geodesic equation enables us to define the connection on the tangent bundle. For any tangent vector field $\psi(t)$ along the path $\vphi(t)$, the covariant derivative along $\vphi(t)$ is defined by
\begin{align*}
D_t\psi=\frac{\p\psi}{\p t}-\frac{1}{2}(\nabla_t\psi,\nabla_t\vphi')_{g_{\vphi}}.
\end{align*}
Then the connection at $\vphi$ is given by
\begin{align*}
\G(\psi_1,\psi_2)=-\frac{1}{2}(\nabla\psi_1,\nabla\psi_2)_{g_{\vphi}}
\end{align*} for any $\psi_1$ and $\psi_2$ in $T_\vphi\mathcal{H}$. Moreover, $\G$ is torsion-free and metric-compatible. The following theorem is proved in  \cite{MR909015}, \cite{MR1736211} and \cite{MR1165352}.
\begin{thm}(Mabuchi \cite{MR909015}, Donaldson \cite{MR1736211}, Semmes \cite{MR1165352})
The Riemannian manifold $\mathcal{H}$ is an infinite dimensional symmetric space; it admits a Levi-Civita connection whose curvature is covariant constant. At a point $\vphi\in\mathcal{H}$ the curvature is given by
\begin{align*}
R_\vphi(\delta_1\vphi,\delta_2\vphi)\delta_3\vphi
=-\frac{1}{4}\{\{\delta_1\vphi,\delta_2\vphi\}_\vphi,
\delta_3\vphi\}_\vphi,
\end{align*} where $\{,\}_\vphi$ is the Poisson bracket on $C^\infty(M)$ of the symplectic form $\om_\vphi$.
\end{thm}
Chen established the following theorem in \cite{MR1863016}.
\begin{thm}(Chen \cite{MR1863016})
The following is true\textup{:}
\begin{enumerate}
\renewcommand{\labelenumi}{(\roman{enumi})}
\item $\mathcal{H}$ is convex by $C^{1,1}$ geodesics.
\item $\mathcal{H}$ is a metric space.
\end{enumerate}
\end{thm}
Later, Calabi and Chen proved $\mathcal{H}$ is negatively curved in the sense of Alexanderof in \cite{MR1969662}.
We denote the space of normalized K\"ahler potentials by
\begin{align}\label{SKM: H0}
\mathcal{H}_0
=\{\vphi\in C^{\infty}(M,R)\vert\om
+\sqrt{-1}\p\bar\p\vphi>0 \text{ and } I(\vphi)=0\},
\end{align}
where
\begin{align*}
I(\vphi)
&=\frac{1}{V}\sum_{p=0}^n\frac{1}{(p+1)!(n-p)!}
\int_M\vphi\om^{n-p} \wedge (\p\bar\p\vphi)^p.
\end{align*}
In fact,
$\mathcal{H}$ can be naturally split as
$$\mathcal{H}=\mathcal{H}_0\times \mathbf{R}.$$
It leads to the decomposition of the tangent space
$$T_\vphi=\{f\vert\int_Mf\om^n_\vphi=0\}\oplus\mathbf{R}.$$
On a K\"ahler-Einstein manifold $(M,\om)$, choose $\om$ be
the reference metric.
It is clear that $h_\om=0$ by the definition \eqref{ricci pot of back}.
Substituting this into the potential equation of K\"ahler-Ricci flow \eqref{KRF potential}, we obtain that
\begin{equation}\label{KRF potential ke}
\begin{cases}
\frac{\p\vphi}{\p t}&=\log\frac{\om^n_\vphi}{\om^n}+\vphi+a(t)\\
\vphi(0)&=\vphi_0.
\end{cases}
\end{equation}
Furthermore, we choose appropriate normalization constant
\begin{align}\label{a(t)}
a(t)=
-\frac{1}{V}\int_M(\log\frac{\om^n_\vphi}{\om^n}
+\vphi)\om^n_\vphi,
\end{align} then one obviously sees that
\begin{align}\label{normalized cond}
\p_tI(\vphi)=\frac{1}{V}\int_M\p_t\vphi\om_\vphi^n=0.
\end{align} We first assume $\vphi_0\in\mathcal{H}_0$ such that $I(\vphi_0)=0$, the general case will be treated in Section \ref{WF}. Then \eqref{normalized cond} implies $I(\vphi)=0$ which assures the solution $\vphi$ of \eqref{KRF potential ke}
always stays in $\mathcal{H}_0$.

For any $\vphi\in\mathcal{H}$, Mabuchi \cite{MR867064} defined the $K$-energy of $(M,\om)$ as follows
\begin{align}\label{K en}
\nu(\om,\om_\vphi)
&=-\frac{1}{V}
\int_0^1\int_{M}\dot\vphi(\tau)(S_{\vphi(\tau)}-\as)
\om^n_{\vphi(\tau)}d\tau\
\end{align}
where $\vphi(\tau)$ is an arbitrary piecewise smooth path from $0$ to
$\vphi$. Later on, the explicit expression
of the $K$-energy is given in Chen \cite{MR1772078} and Tian \cite{MR1787650} as
\begin{align}\label{K en im}
\nu_\om(\vphi)
&=\frac{1}{V}\int_M\log\frac{\om^n_\vphi}{\om^n}\om_\vphi^{n}
+\frac{\as n!}{V}I(\vphi)\nonumber\\
&-\frac{1}{V}\sum_{i=0}^{n-1}\frac{n!}{(i+1)!(n-i-1)!}\int_{M}\vphi
Ric\wedge\om^{n-1-i}\wedge(\p\bar\p\vphi)^{i}.
\end{align}
We will in later section simply denote $\nu(\vphi)$ instead of $\nu_\om(\vphi)$.
The second variation of the $K$-energy is given in Mabuchi \cite{MR909015}.
\begin{thm}(Mabuchi \cite{MR909015})
If $\om$ is a critical point of $\nu(\vphi)$, then the inequality
\begin{align}
\frac{d^2}{dt^2}\nu(\theta_t)|_{t=0}\geq0
\end{align} holds for every smooth path $\{\theta_t|-\eps\leq t\leq\eps\}$ in $\mathcal{K}$ such $\theta_0=\om$.
\end{thm}
Let $Aut(M)$ be the
group of holomorphic automorphisms of $M$ and $Aut_0(M)$ be its identity component.
Bando-Mabuchi \cite{MR946233} and Bando \cite{MR887939} further showed that
\begin{thm}(Bando-Mabuchi \cite{MR946233}, Bando \cite{MR887939})\label{BM en}
Assume $\mathcal{E}\neq \phi$. Then
\begin{enumerate}
\renewcommand{\labelenumi}{(\roman{enumi})}
\item $K$-energy is bounded from below on $\mathcal{K}$ and takes its absolute minimum exactly on $\mathcal{E}$.
\item $\mathcal{E}$ consists a single $Aut_0(M)$ orbit.
\end{enumerate}
\end{thm}
Indeed the normalization constant $a(t)$ can be estimated by the $K$-energy.
\begin{lem}Let $\vphi$ be the solution of \eqref{KRF potential ke}.
The relation between $a(t)$ and the $K$-energy $\nu(\vphi)$ is given by
\begin{align}\label{a}
a(t)+\nu(\vphi)=a(0)+\nu(\vphi_0).
\end{align}
\end{lem}
\begin{proof}
We calculate the evolution of $a(t)$ along the K\"ahler-Ricci flow directly,
\begin{align*}
V\frac{d}{dt}a(t)
&=-\int_M(\tri_\vphi+1)\dot\vphi\om^n_\vphi
-\int_M(\log\frac{\om^n_\vphi}{\om^n}+\vphi)
\tri_\vphi\dot\vphi\om^n_\vphi.
\end{align*}
According to the Stokes' theorem and \eqref{normalized cond} the first term vanishes identically. Meanwhile, by using the integration-by-part formula and \eqref{KRF potential ke}, the second term becomes
\begin{align*}
\int_M(S_\vphi-n)\dot\vphi\om^n_\vphi.
\end{align*}
Since \eqref{K en} implies $$\frac{d}{dt}\nu(t)
=-\frac{1}{V}\int_M(S_\vphi-n)\dot\vphi\om^n_\vphi,$$ we obtain
\begin{align}\label{pt a}
\frac{d}{dt}a(t)=-\frac{d}{dt}\nu(t).
\end{align}
Thus, the assertion follows
by integrating both sides of \eqref{pt a} with respect to $t$.
\end{proof}
Since the $K$-energy is decreasing along the K\"ahler-Ricci flow, according to \thmref{BM en} we immediately conclude that:
\begin{cor}\label{a(t) bdd}
On a K\"ahler-Einstein manifold,
$a(t)$ is uniformly bounded along the K\"ahler-Ricci flow.
\end{cor}

The following theorems including the
short time existence, the regularity and the continuous dependence on initial data of
the K\"ahler-Ricci flow have been proved in Chen-Ding-Zheng \cite{zheng-2009}, in which
they defined a new second order Monge-Amp\`{e}re flow
called pseudo-Calabi flow
\begin{equation}\label{PCF}
  \left\{
   \begin{aligned}
   {{\partial \varphi}\over {\partial t}}&= -f(\varphi),  \\
   \triangle_\varphi f(\varphi) &= S(\varphi) - \ul S.  \\
   \end{aligned}
  \right.
\end{equation} The pseudo-Calabi flow coincides with the K\"ahler-Ricci flow,
when the initial datum is restricted in the canonical K\"ahler class.
Let $X=C^0([0,T),C^{2+\a}(M,g))\cap
C^1([0,T),C^{\a}(M,g))$.
\begin{thm}(Chen-Ding-Zheng \cite{zheng-2009})\label{short time: main}
Let $\vphi_0\in C^{2,\a}(M,g)$ be such that $\l \om
\leq \om_{\vphi_0}\leq \L \om$ for two positive constants $\l$ and $\L$.
Then
the pseudo-Calabi flow
has a unique solution $\vphi(x,t)\in X$,
where $T$ is the maximal existence time.
\end{thm}
\begin{thm}(Chen-Ding-Zheng \cite{zheng-2009})\label{regularity of PCF: reg of PCF}
The solution of the pseudo-Calabi flow $\vphi\in X$ is smooth for any $t>0$.
\end{thm}
\begin{thm}(Chen-Ding-Zheng \cite{zheng-2009})\label{con dpd: main}
If $\phi$ is the solution of the pseudo-Calabi flow for initial datum $\phi_0$ on $[0,T]$, then there is a neighborhood $U$ of $\phi_0$ such that the pseudo-Calabi flow has a solution $\vphi(t)$ on $[0,T]$ for any $\vphi_0\in U$ and the mapping $\vphi_0\mapsto\vphi(t)$ is $C^k$ for $k=0,1,2,\ldots$
\end{thm}
A direct corollary of the continuous dependence on initial data \thmref{con dpd: main} says,
\begin{thm}(Chen-Ding-Zheng \cite{zheng-2009})\label{short time: stability}
If $M$ admits a cscK metric $\om$. Let $\vphi_0\in C^{2,\a}(M,g)$ be such that $\l \om
\leq \om_{\vphi_0}\leq \L \om$ for two positive constants $\l$ and $\L$.
Then for any $T>0$
there exits a positive constant $\eps_0(T)$. If
$|\vphi_0|_{C^{2,\a}(M,g)}\leq\eps_0(T)$,
then the pseudo-Calabi flow has a unique solution on $[0,T]$, and
$$|\dot\vphi|_{C^{\a}(M,g)}+|\vphi|_{C^{2,\a}(M,g)}\leq C\eps_0(T)$$
for all $t\in[0,T]$, where $C$ depends on $M$, $g$ and $T$. Furthermore
$\eps_0(T)$ goes to zero, as $T$ goes to infinity.
\end{thm}
\section{No nontrivial holomorphic vector fields}\label{Nnhvf}
Let $\eta(M)$ be the set composed of all holomorphic vector fields on $M$.
Now, we start with the case $\eta(M)=\phi$. We shall prove the following proposition in this section.
\begin{prop}\label{no hol}
Assume $M$ admits a K\"ahler-Einstein metric $\om$ and has no
holomorphic vector fields.
There exits a small positive constant $\eps_0$,
suppose the initial datum satisfies $$|\vphi_0|_{C^{2,\a}(M)}\leq\eps_0,$$
then the K\"ahler-Ricci flow $g_\vphi$ converges smoothly to $g$.
\end{prop}
\begin{proof}
We at first show that under the assumption of the proposition,
the solution of \eqref{KRF potential ke} always stays in some small $\eps_1$-neighborhood of the zero function.
\begin{lem}
For any $\eps_1>0$, there exits a small positive constant $\eps_0$.
If $$|\vphi_0|_{C^{2,\a}(M)}\leq\eps_0,$$
then $|\vphi(t)|_{2,\a}\leq\eps_1$ for all
$t\in[0,+\infty)$.
\end{lem}
\begin{proof}
Suppose that the conclusion fails, then there must exist a sequence of initial datum
$\vphi^0_{s}$ such that
\[
|\vphi^0_s|_{C^{2,\a}}\leq\frac{1}{s}.
\]
By virtue of \thmref{short time: stability},
we get a sequence of solutions $\vphi_s(t)$ satisfying the
flow equations \eqref{KRF potential ke}
with $\vphi_s(0)=\vphi^0_s$.
Let $T_s$ be the first time such that
\begin{align}\label{no ho ass}
|\vphi_s(T_s)|_{C^{2,\a}}=\eps_1
\text{ and }|\vphi_s(t)|_{C^{2,\a}}<\eps_1
\end{align}
on $[0,T_s)$. According to \thmref{short time: stability} again, we have $T_s\geq T_1>0$.
Moreover, we apply \thmref{regularity of PCF: reg of
PCF} to \eqref{KRF potential ke} on $[T_s-2a,T_s]$ for fixed $a$ such that $0<a<\frac{T_s}{2}-\frac{T_1}{4}$, then we obtain the uniform higher order bound of the sequence of
the solutions
\[
|\vphi_s|_{C^{k,\a}(M)}\leq C(k,\eps_1,a), \forall k\geq0
\] on $[T_s-a,T_s]$.
Consequently, there is a subsequence of $\phi_s=\vphi_s(T_s)$ converges smoothly to $\phi_\infty$ satifying
\begin{align}\label{no ho ass lim}
|\phi_\infty|_{C^{2,\a}}=\eps_1.
\end{align}
It is obvious that $g_{\phi_\infty}$ is still a K\"ahler metric.
Since the $K$-energy is not only well defined for $\vphi^0_s$ by \eqref{K en im} but also decreasing along the K\"ahler-Ricci flow,
\thmref{BM en} implies
\[
0\leq\nu_\om(\phi_s)\leq\nu_\om(\vphi_s(0))\leq\frac{C}{s}.
\]
By passing the limit we obtain
\[
\lim_{s\rightarrow\infty}\nu_\om(\vphi_s)
=\nu_\om(\vphi_\infty)=0.
\]
According to \thmref{BM en}, we obtain $g_{\phi_\infty}$ is a K\"ahler-Einstein metric.
From \thmref{BM en} we deduce that
$\phi_\infty$ must be a constant. Furthermore the normalization condition $I(\phi_\infty)=0$ gives rise to $\phi_\infty=0$ which
contradicts to \eqref{no ho ass lim} and the lemma follows.
\end{proof}
According to \thmref{regularity of PCF: reg of
PCF} and $|\vphi(t)|_{C^{2,\a}}\leq\eps_1$ uniformly, we
have that $|\vphi(t)|_{C^k}\leq C_k$ for any $k\geq3$ away from
$t=0$. It follows that there is a subsequence of
any sequence $t_i$ converges smoothly to a limit function
$\vphi_\infty$. Moreover,
since the $K$-energy has lower bound
and it decays along the flow,
$\om_{\vphi_\infty}$ must be a K\"ahler-Einstein metric.
This togethers with \thmref{BM en} and the normalization condition implies that $\vphi_\infty=0$. Because ${t_i}$ is chosen randomly, we conclude the K\"ahler-Ricci flow converges smoothly
to the original K\"ahler-Einstein metric.
\end{proof}
\section{$M$ admits nontrivial holomorphic vector fields}\label{NHVF}
\subsection{Choice and estimate of holomorphic transformations}\label{CHT}
When M admits holomorphic vector fields,
we need to find an appropriate
holomorphic transformation.
Set $\mathcal{E}_0\subset\mathcal{H}_0$
be the space of K\"ahler potentials of K\"ahler-Einstein metrics.

Let $\sigma^\ast_t\om$ be any curve  with $\sigma_0=id$
in $\mathcal{E}_0$,
the tangent vector at $\om$ is
$\frac{d}{dt}\sigma^\ast_t|_{t=0}\om=L_X\om$. Here $X=(\sigma_t)_\ast^{-1}\p_t\sigma_t|_{t=0}$ is
the real part of some holomorphic vector field. Since $C_1(M)>0$ implies $M$
is simple connected by Kobayashi \cite{MR0133086}, we obtain
$L_X\om=\sqrt{-1}\p\bar\p\theta_X$ for some function $\theta_X$.
Hence, that the dimension of
the space of holomorphic vector fields is finite
infers which of $\mathcal{E}_0$ is finite.
Moreover, according to Mabuchi \cite{MR909015}, we have $\mathcal{E}_0$ is also a totally geodesic submanifold
of $\mathcal{H}_0$.
Then the point $\rho\in \mathcal{E}_0$ realizes the shortest distance between
$\varphi$ and $\mathcal{E}_0$ is uniquely determined.
In fact, according to \thmref{BM en}
we obtain a holomorphic diffeomorphism $\sigma\in Aut_0(M)$
such that
$\sigma^\ast\om
=\om+\sqrt{-1}\p\bar\p\rho$.
The following invariance of the $K$-energy
under the holomorphic transformation is known in Mabuchi \cite{MR867064}.
\begin{lem}\label{PCF on SoKM: energy decay}
$\nu(\om,\om_{(\sigma^{-1})^\ast(\vphi-\rho)})=\nu(\om,\om_{\vphi})=
\nu(\om_{\rho},\om_{\vphi})$.
\end{lem}
\begin{proof}
Since $\omega$ and $\omega_{\rho}$ are both K\"ahler-Einstein metrics,
we have that
\begin{align*}
\nu(\om,\om_{(\sigma^{-1})^\ast(\vphi-\rho)})
&=\nu(\sigma^\ast\om,\om_{\vphi})
=\nu(\om_{\rho},\om_{\vphi})\\
&=\nu(\om_{\rho},\om)+\nu(\om,\om_{\vphi})
=\nu(\om,\om_{\vphi}).
\end{align*}
The first inequality holds by Lemma (5.4.1) in \cite{MR867064}
and the last one follows by Theorem (5.3) in \cite{MR867064}.
\end{proof}
In \cite{zheng-2009} we prove some lemmas
regarding to the metric geometry
of the space of K\"ahler-Einstein metrics.
The following two lemmas show that
when metrics stay close to $\om$,
their projection metrics are uniformly bounded.

\begin{lem}\label{Kf on SoKM: dis imp norm}
There exists a positive constant $\eps$,
for any $\rho$ satisfies $d(0,\rho)\leq\eps$,
we have
$|\rho|_{C^{3,\a}}\leq C_2\eps$.
\end{lem}
\begin{proof}
Since $\mathcal{E}_0$ is a finite dimension Riemannian symmetric space,
the small $\eps$ neighborhood near $\rho=0$ in this submanifold can be pulled back
by the exponential map $exp_0$ to the tangent space
$T_0(\mathcal{E}_0)$ at $0$.
Denote $\psi=exp_0^{-1}(\rho)$. Then the length from $\psi$ to $0$
is $\eps$.
We obtain the norm induced by the distance on $T_0(\mathcal{E}_0)$ is equivalent to the $C^{2,\a}$ norm, since all norms on a finite-dimensional vector space are equivalent.
Then we have $|exp_0^{-1}(\rho)|_{C^{2,\a}}$ is bounded by $C_1\eps$.
Furthermore, since the exponential map is a diffeomorphism in the $\eps$ neighborhood near $\rho=0$,
we obtain $|\rho|_{C^{2,\a}}\leq C_2\eps$ for some constant $C_2$ and this lemma follows by an appropriate choice of $\eps$.
\end{proof}
\begin{rem}
In fact, we can improve the above conclusion
in \lemref{Kf on SoKM: dis imp norm} for $C^{k}$
of fix $k\geq0$, not only for $C^{3,\a}$ norm.
\end{rem}
\begin{lem}\label{Kf on SoKM: gauge bound}
There exists a positive constant $\eps_1$.
If $|\vphi|_{C^{2,\a}}\leq\eps_1$,
then $|\rho|_{C^{3,\a}}\leq C_4$
and $|\sigma|_h\leq C_5$.
Here $h$ is the left invariant metric in $Aut(M)$.
\end{lem}
\begin{proof}
Choose a path $\gamma_t=t\vphi-I(t\vphi)\in \mathcal{H}_0$
for $0\leq t\leq1$.
Denote $d(0,\vphi)$
the distance between $0$ and $\vphi$.
Then by using \eqref{SKM: length}, we compute
\begin{align*}
d(0,\vphi)
&\leq L(\gamma_t)=\int_0^1(\int_M(\frac{\p\gamma_t}{\p t})^2
\om^n_{\gamma_t})^{\frac{1}{2}}dt\\
&=\int_0^1(\int_M(\vphi-\p_tI(t\vphi))^2
\om^n_{\gamma_t})^{\frac{1}{2}}dt\leq C_3\eps_1
\end{align*}
for $|\vphi|_{C^{2,\a}}\leq\eps_1$.
Moreover,
the choice of the $\rho$ implies
\begin{align*}
d(0,\rho)
\leq d(0,\varphi)+d(\varphi,\rho)
\leq 2d(0,\varphi) \leq C_3\epsilon_1
\end{align*}
by the triangle inequality.
From \lemref{Kf on SoKM: dis imp norm}, it follows that
$|\rho|_{C^{3,\a}}\leq C_4=C_2C_3\eps_1$.
Furthermore, using Lemma 4.6 in Chen-Tian \cite{MR2219236},
we derive $|\sigma|_h\leq C_5$ and the lemma follows.
\end{proof}
\begin{rem}
Alternatively the
holomorphic transformation can be derived by minimizing $I-J$ functional
in Bando-Mabuchi's work \cite{MR946233}, that will be further discussed in the Section \ref{ACHT}.
They use this minimizer to prove the
uniqueness of the K\"ahler-Einstein metric
when the first Chern class is positive.
The minimizer of $I-J$ functional is not unique in general,
since the second variation of this functional is not strictly positive.
However, we observe that when the potential is small enough,
the minimizer is unique. Furthermore, we prove a similar estimate \corref{PCF on SoKM: gauge bound} to \lemref{Kf on SoKM: gauge bound}.
\end{rem}
\subsection{Long time existence and Cheeger-Gromov convergence}
Set
$$\mathcal{S}(\eps_1,C(k,\eps_1))
=\{\vphi\vert|\vphi|_{C^{2,\a}}\leq\eps_1;
|\vphi|_{C^{k,\a}(M)}\leq C(k,\eps_1)\}.$$
It is obvious that $0\in\mathcal{S}$.
\begin{lem}\label{PCF on SoKM: norm bound by enery}
For any $\eps>0$, There is a small positive constant $o$ depends on $\eps$ and $\mathcal{S}$
such that for any $\vphi\in\mathcal{S}$,
if $\nu_{\om}(\vphi)\leq o$,
then $|(\sigma^{-1})^\ast(\vphi-\rho)|_{C^{2,\a}}<\eps$.
\end{lem}
\begin{proof}
 If the conclusion fails, we assume there exist a positive constant $\eps$ and a sequence of
 $\vphi_s\in\mathcal{S}$ satisfying
\begin{align}\label{PCF on SoKM: bound of vs}
\nu_{\om}(\vphi_s)\leq\frac{1}{s}\nonumber
\end{align} such that
\begin{equation}\label{PCF on SoKM: contra assum}
|(\sigma_s^{-1})^\ast(\vphi_s-\rho_s)|_{C^{2,\a}}\geq\eps.
\end{equation}
Since $\vphi_s\in \mathcal{S}(\eps_1,C(k,\eps_1))$, we obtain a subsequence $\vphi_{s_j}$ of $\vphi_s$ converges to $\vphi_\infty$ in smooth sense.
Let $\hat\vphi_s=(\sigma_s^{-1})^\ast(\vphi_s-\rho_s)$.
\lemref{Kf on SoKM: gauge bound} gives
\begin{align*}
|\rho_s|_{C^{3,\a}}\leq C_4\text{ and }|\sigma_s|_h\leq C_5
\end{align*}
which implies there are subsequences (using the same notation) of
$\rho_{s_j}$ and $\sigma_{s_j}$
by the Azela-Ascoli theorem and the Bolzano-Weierstrass theorem respectively
such that
\begin{align*}
&\rho_{s_j}\rightarrow\rho_\infty \text{ in } C^{3,\b} \text{ sense for any } \b<\a \\&\text{ and }
\sigma_{s_j}\rightarrow\sigma_\infty \text{ in the left invariant metric}.
\end{align*}
Then combining with \lemref{PCF on SoKM: energy decay} which
implies that
$$\nu_{\om}(\vphi_\infty)=\nu_{\om}(\hat\vphi_\infty)=0$$ we derive $\hat\vphi_{s_j}$ converges to $\hat\vphi_\infty=(\sigma_\infty^{-1})^\ast(\vphi_\infty-
\rho_\infty)\in\mathcal{E}_0$ in $C^{3,\b}$ and $\sigma_\infty^\ast\om=
\om+\p\bar\p\rho_\infty$. Moreover, according to \thmref{BM en}, we have $\hat\vphi_\infty,\vphi_\infty\in \mathcal{E}_0$.
We claim that $$d(\vphi_\infty,\rho_\infty)=0.$$
Otherwise for some sufficient large $N$,
when $s_j>N$,
$d(\vphi_{s_j},\rho_{s_j})=d(\vphi_{s_j},\mathcal{E}_0)$ has a strictly positive lower bound.
Since it is shown the distance function is at least $C^1$ in Chen \cite{MR1863016},
we have $d(\vphi_\infty,\mathcal{E}_0)>0$
that contradicts to $\vphi_\infty\in \mathcal{E}_0$.
Consequently, this claim holds and
implies $\hat\vphi_\infty=0$ which is a contradiction to
$|\hat\vphi_\infty|_{C^{2,\a}}\geq\eps$ given by \eqref{PCF on SoKM: contra assum}.
\end{proof}

\begin{prop}\label{hol}
Assume $M$ admits a K\"ahler-Einstein metric $\om$
and has nontrivial holomorphic vector fields.
There is a small positive constant $\eps_0$.
If $|\vphi_0|_{C^{2,\a}(M)}\leq\eps_0$,
then there is a unique solution $\vphi(t)$
and the corresponding holomorphic transformation $\varrho(t)$
such that the normalization potential of $\varrho(t)^\ast \om(t)$
always stays in $\mathcal{S}$.
Moreover, for any sequence $\om(t_i)$,
there is a subsequence $\om(t_{i_j})$ such that
$\varrho(t_{i_j})^\ast \om(t_{i_j})$ converges smoothly to a K\"ahler-Einstein metric $\om_\infty$.
\end{prop}

\begin{figure}[t]
\begin{center}
  \includegraphics[width=\columnwidth]{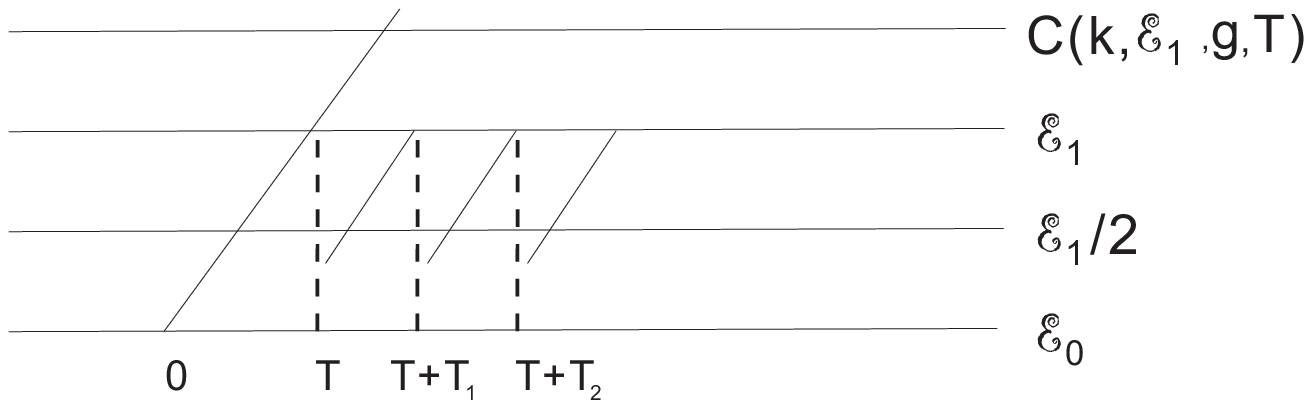}
\end{center}
\caption{Illustration of the proof idea of Proposition \ref{hol} : solving the equation after pulling-back.}
\label{F:sta}
\end{figure}

\begin{proof}
We prove this proposition by the contradiction method.
Let $\eps_1$ be determined in \lemref{Kf on SoKM: gauge bound}.
Owing to \thmref{short time: stability},
we assume there is a maximal time $T$ such that
$$|\vphi|_{C^{2,\a}}<\eps_1\text{ on }[0,T)\text{ and }
|\vphi(T)|_{C^{2,\a}}=\eps_1.$$
According to \thmref{regularity of PCF: reg of PCF}
we obtain $|\vphi(T)|_{C^{k,\a}}\leq C(k,\eps_1,\frac{T}{2})$ on $[\frac{T}{2},T]$. So we get $$\vphi(T)\in\mathcal{S}(\eps_1,C(k,\eps_1,\frac{T}{2})).$$
There are two situations.
If $\vphi(T)$ is a K\"ahler-Einstein metric,
the flow will stop
here and our theorem is proved.
Otherwise, we will extend the flow as follows.
We first choose $\eps_0$ small enough to guarantee
$$\nu_\om(\vphi_0)\leq o(\frac{\eps_1}{2},\mathcal{S}(\eps_1,C(k,\eps_1,\frac{T}{2})))$$
where the constant $o(\frac{\eps_1}{2},\mathcal{S}(\eps_1,C(k,\eps_1,\frac{T}{2})))$ is determined
in \lemref{PCF on SoKM: norm bound by enery}.
Let the holomorphic transformation $\sigma$ be the projection of $\vphi(T)$ in $\mathcal{E}_0$ with
$\sigma^\ast\om
=\om+\sqrt{-1}\p\bar\p\rho\nonumber$. We set $\vphi_1^0$ be the K\"ahler potential of the metric pulled back by $\sigma$, i.e.
$$(\sigma^{-1})^\ast\om_{\vphi(T)}
=\om+\sqrt{-1}\p\bar\p[(\sigma^{-1})^\ast(\vphi(T)-\rho)]
=\om+\sqrt{-1}\p\bar\p\vphi_1^0.$$
Since the $K$-energy is decreasing along the K\"ahler-Ricci flow, we obtain according to \lemref{PCF on SoKM: energy decay},
\begin{align}\label{PCF on SoKM: energy tra}
\nu_\om(\vphi_1^0)\leq o(\frac{\eps_1}{2},\mathcal{S}(\eps_1,C(k,\eps_1,\frac{T}{2}))).
\end{align}
So \lemref{PCF on SoKM: norm bound by enery} implies that
\begin{align}\label{PCF on SoKM: itr small norm}
|\psi|_{C^{2,\a}(g)}
=|(\sigma^{-1})^\ast(\vphi(T)-\rho)|_{C^{2,\a}(g)}
<\frac{\eps_1}{2}.
\end{align}
We next show that the K\"ahler-Ricci flow is invariant under the transformation. Let $\vphi_1=
(\sigma^{-1})^\ast(\vphi(t)-\rho)$. Since
\begin{align*}
\frac{\p}{\p t}\vphi_1
&=(\sigma^{-1})^\ast[\log\frac{\om^n_{\vphi}}{\om^n}+\vphi-\frac{1}{V}\int_M(\log\frac{\om^n_\vphi}{\om^n}
+\vphi)\om^n_\vphi]\\
&=(\sigma^{-1})^\ast[\log\frac{\om^n_{\vphi}}{\om_\rho^n}
+\vphi-\rho-\frac{1}{V}\int_M(\log\frac{\om^n_\vphi}{\om_\rho^n}
+\vphi-\rho)\om^n_\vphi]\\
&=\log\frac{\om^n_{\vphi_1}}{\om^n}
+\vphi_1-\frac{1}{V}\int_M(\log\frac{\om^n_{\vphi_1}}{\om^n}
+\vphi_1)\om^n_{\vphi_1}]
\end{align*}
The second equality follows form the fact that $\om_{\rho}$ is a K\"ahler-Einstein metric.
We conclude that
$\vphi_1$
is the solution of the equation of the form
\begin{equation}\label{PCF on SoKM: extend equ}
\begin{cases}
\frac{\p}{\p t}\vphi_1
&=\log\frac{\om^n_{\vphi_1}}{\om^n}+\vphi_1+a(t)\\
\vphi_1(0)&=\vphi_1^0=(\sigma^{-1})^\ast(\vphi(T)-\rho)
\end{cases}
\end{equation}
with \eqref{PCF on SoKM: itr small norm} and \eqref{PCF on SoKM: energy tra}.
Again, \thmref{short time: stability}
implies \eqref{PCF on SoKM: extend equ} has a solution on $[0,T_1]$ with $T_1\geq T$ such that $$|\vphi_1(T_1)|_{C^{2,\a}}=\eps_1.$$
Moreover, let $\vphi(t)=\sigma^\ast\vphi_1(t-T)+\rho$ on $[T,T+T_1)$,
the new $\vphi(t)$
is the solution of \eqref{KRF potential ke} on $[0,T+T_1]$.
Then we repeat the same steps inductively for
\begin{equation*}
\vphi_{s-1}(T_{s-1})\in\mathcal{S}(\eps_1,C(k,\eps_1,\frac{T}{2})),
\end{equation*}
with $T_{s-1}\geq T$ which is obtained in \thmref{short time: stability}
till $\vphi_s$ becomes a K\"ahler-Einstein at time $T_s$, if $T_s<\infty$.
If not, we have the K\"ahler-Ricci flow has the long time existence and the solution $\vphi(t)$ for all $t\geq0$ given by
$$\om_{\vphi(t)}
=\prod_{i=0}^{s-1}\sigma_{i}^\ast\om_{\vphi_{s}(t)}\text{ on } [\sum_{i=0}^{s-1}T_{i},\sum_{i=0}^{s}T_{i}).$$
For any sequence $\{\vphi_{t_j}\}$,
there is $s$ such that
$\sum_{i=0}^{s-1}T_{i}\leq t_j \leq\sum_{i=0}^{s}T_i$.
Furthermore, let $\varrho_{j}=(\prod_{i=0}^{s-1}\sigma_{i})^{-1}$. We have
\begin{align*}
|\varrho_{j}^\ast\om_{\vphi_{t_j}}
-\om|_{C^{\a}}\leq \eps_1 \text{ and }
|\varrho_{j}^\ast\om_{\vphi_{t_j}}
-\om|_{C^{k}}\leq C(k,\eps_1,\frac{T}{2}).
\end{align*}
Therefore all metrics are equivalent and their derivatives are bounded. We denote
$$\om_{\psi_{t_j}}=\varrho_{j}^\ast \om_{\vphi_{t_j}}.$$
It follows that by abuse of notation there is a subsequence of
$\om_{\psi_{t_j}}$
converges to a limit metric $\om_\infty$. However, $\om_\infty$
depends on the choice of the subsequence.
Since the $K$-energy is bounded below, we have
$\lim_{s\rightarrow\infty}\nu(\om,\om_{\psi_{t_j}})=0$.
It follows that $g_\infty$ is a K\"ahler-Einstein metric from \thmref{BM en}.
Consequently, this proposition is proved.
\end{proof}
Let $t_s=\sum_{i=0}^{s}T_i$.
Follow the same argument in Chen-Tian \cite{MR2219236},
we can first connect each disperse points by geodesics
in the space of K\"ahler-Einstein metric so that
$$\varrho(t)=\varrho(s)\exp((t-s)X_s),\forall t\in[s,s+1]$$
for $X_s$ is uniformly bounded by \lemref{Kf on SoKM: gauge bound}.
Then smooth the corner at each points $t_s$ by replacing the
broken line by a smooth curve in small neighborhood near $t_s$
without changing the value and $t$ derivative in the end points.
Hence we have extended the holomorphic transformation to each $t$
so that it is Lipschitz continuous in $t$.

\section{Exponential convergence}\label{EC}
In this section, we show that the sequence
of holomorphic transformations $\varrho(t)$ are compact and the exponential convergence of the K\"ahler-Ricci flow. Let $\om_{\psi(t)}=\varrho(t)^\ast \om_{\vphi(t)}.$
We have already obtained
$\lim_{t\rightarrow\infty}Ric(g_{\psi_{t}})-\om_{\psi_{t}}=0$ in Section \ref{NHVF}.
Since the holomorphic transformation keeps this identity invariant, we obtain
\begin{align}\label{ricci}
\lim_{t\rightarrow\infty}Ric(g_{\vphi_{t}})-\om_{\vphi_{t}}=0.
\end{align}
After taking the trace, we get
\begin{align}\label{scalar}
\lim_{t\rightarrow\infty}\tri_{\vphi_t}\dot\vphi
=\lim_{t\rightarrow\infty}S(g_{\vphi_t})-n=0.
\end{align}
Now for each time $t$, we apply the De Giorgi iteration
to derive the $L^\infty$ bound of $\dot\vphi$
under the normalization condition \eqref{normalized cond}. In the following, the constant $C$ may be different from line to line.
\begin{lem}\label{dot vphi}Along the K\"ahler-Ricci flow, we have
\begin{align}
||\dot\vphi||_\infty\leq C||S-n||_\infty \text{ for any } t>0.
\end{align}
\end{lem}
\begin{proof}
We first notice that $\dot\vphi$ satisfies the equation
\begin{align*}
\tri_{\vphi}\dot\vphi=-S(g_{\vphi_t})+n.
\end{align*}
Then we multiply this equation with $\dot\vphi^{2q-1}$ and integrate on $M$ to get
\begin{align}\label{dvphi it}
\frac{2q-1}{q^2}
\int_{M}|\nabla(\dot\vphi^q)|^2\om^n_\vphi
=\int_{M}\dot\vphi^{2q-1}(S-n)\om^n_\vphi.
\end{align}
Since $\om_\vphi=\sigma^\ast\om_\psi$ and
the Sobolev constant and Poincare constant are invariant under the holomorphic transformation, we have they are uniformly bounded. Let $q=1$, we have
\begin{align}\label{q=1}
\int_{M}|\nabla\dot\vphi|^2\om^n_\vphi
=\int_{M}\dot\vphi(S-n)\om^n_\vphi.
\end{align}
We apply the Poincare inequality and \eqref{normalized cond} to \eqref{q=1} to get
\begin{align*}
||\dot\vphi||^2_{L^{2}}
\leq C_p^2\int_{M}\dot\vphi(S-n)\om^n_\vphi.
\end{align*}
Here $C_p$ is the Poincare constant.
Then we use the $\frac{1}{2}$-H\"older inequality to get
\begin{align}\label{dvphi L2}
||\dot\vphi||^2_{L^{2}}
&\leq \frac{1}{2}||\dot\vphi||^2_{L^{2}}+\frac{1}{2}C^4_p||S-n||^2_\infty\nonumber\\
&\leq C^4_pC(\delta)||S-n||^2_\infty
\end{align}
and apply the $\frac{1}{2}$-H\"older inequality to \eqref{q=1} again to get
\begin{align*}
\int_{M}|\nabla\dot\vphi|^2\om^n_\vphi
\leq||\dot\vphi||^2_{L^2}+||S-n||^2_\infty.
\end{align*}
Accordingly, combining with \eqref{dvphi L2} we obtain
\begin{align}\label{dvphi 1,2}
||\dot\vphi||_{W^{1,2}}
\leq C||S-n||_\infty.
\end{align}
Let $p>n$.
We claim: \begin{align}\label{dvphi q}
||\dot\vphi||_{\frac{2np}{2n+p}}
\leq C||S-n||_\infty.
\end{align}
When $n=1$, the claim obviously follows from \eqref{dvphi 1,2} by the Sobolev imbedding theorem.
Since the Sobolev constant is bounded, we can use the Sobolev inequality
\begin{align}\label{solblev inequ}
C_1(\int_M|f|^{\frac{2n}{n-1}}\om_\vphi^{n})^\frac{n-1}{n}
\leq C_2\int_M|f|^2\om_\vphi^{n}+\int_M|\nabla f|^2\om_\vphi^{n}.
\end{align}
For any $n\geq 2$,
we apply the Sobolev inequality \eqref{solblev inequ} to the left side of \eqref{dvphi it} and the H\"older inequality to the right one to obtain
\begin{align*}
(\int_M|\dot\vphi|^{\frac{2nq}{n-1}}
\om_\vphi^{n})^\frac{n-1}{n}
\leq C\int_M|\dot\vphi|^{2q}\om^n_\vphi
+\int_M(S-n)^{2q}\om^n_\vphi.
\end{align*}
Setting $q=\frac{(n-1)p}{2n+p}$, we have
\begin{align}\label{dvphi claim}
(\int_M|\dot\vphi|^{\frac{2np}{2n+p}}
\om_\vphi^{n})^\frac{n-1}{n}
\leq C\int_M|\dot\vphi|^{\frac{2(n-1)p}{2n+p}}\om^n_\vphi
+\int_M(S-n)^{\frac{2(n-1)p}{2n+p}}\om^n_\vphi.
\end{align}
Since the $L^p$ interpolation inequality implies
\begin{align*}
||\dot\vphi||_{\frac{2(n-1)p}{2n+p}}\leq \eps||\dot\vphi||_{\frac{2np}{2n+p}}
+C(\eps)||\dot\vphi||_{2},
\end{align*}
we have by \eqref{dvphi 1,2}
\begin{align}\label{dvphi p int}
||\dot\vphi||_{\frac{2(n-1)p}{2n+p}}\leq \eps||\dot\vphi||_{\frac{2np}{2n+p}}
+C(\eps)||S-n||_\infty.
\end{align}
Hence substituting \eqref{dvphi p int} into \eqref{dvphi claim} we obtain the claim provided $C\eps=\frac{1}{2}$.

Next, let $x$ be any point in $M$ and $\eta$ be a smooth cut-off function defined in the closed ball $B_{2r}(x)$ centered at the point $x$ such that $\eta$ equals to $1$ within $B_{\frac{1}{2}r}(x)$ and vanishes outside $B_{r}(x)$. In $B_{2r}(x)$, by chain and product rules $\eta\dot\vphi$ satisfies
\begin{align}\label{eta vphi}
\tri_{\vphi}(\eta\dot\vphi)&=(\eta-2\tri\eta)(-S(g_{\vphi_t})+n)
+\tri_\vphi\eta\dot\vphi
+2\nabla^i(\nabla_i\eta(-S(g_{\vphi_t})+n))\nonumber \\
&\doteq f+\nabla^i g_i
\end{align}
for $f=(\eta-2\tri\eta)(-S(g_{\vphi_t})+n)
+\tri_\vphi\eta\dot\vphi$ and $g=2(\nabla_i\eta(-S(g_{\vphi_t})+n))$.
Set $F=(||f||_{\frac{2np}{2n+p}}+||g_i||_{p})$ and the test function $\psi$ be $\max\{k,\eta\dot\vphi\}$ for any fixed $k>0$. Multiplying \eqref{eta vphi} with $\psi$ and integrating by part on $M$, we have
\begin{align}\label{dvphi local}
\int_M|\nabla\psi|^2\om_\vphi^n=-\int_M f\psi\om_\vphi^n+\int_M g_i\psi^i\om_\vphi^n.
\end{align}
By using the H\"older inequality to the right side of \eqref{dvphi local} we obtain
\begin{align*}
||\nabla\psi||^2_{2}\leq F(||\nabla\psi||_{2}+||\psi||_{\frac{2n}{n-1}})|A(k)|^{\frac{1}{2}-\frac{1}{p}}
\end{align*}
for $A(k)=\{y\in B_r(x)\vert(\eta\dot\vphi)(y)>k\}$.
Using the Sobolev inequality \eqref{solblev inequ} to $\psi$, we have
\begin{align}\label{solblev inequ loc}
C_1(\int_M|\psi|^{\frac{2n}{n-1}}\om_\vphi^{n})^\frac{n-1}{n}
&\leq C_2\int_M|\psi|^2\om_\vphi^{n}+\int_M|\nabla \psi|^2\om_\vphi^{n}\nonumber\\
&\leq C_2|B_r|^{\frac{1}{n}}
(\int_M|\psi|^{\frac{2n}{n-1}}\om_\vphi^{n})^\frac{n-1}{n}
+\int_M|\nabla \psi|^2\om_\vphi^{n}\nonumber\\
&\leq 2\int_M|\nabla \psi|^2\om_\vphi^{n}
\end{align} when we choose $r$ small enough so that $C_2|B_r|^{\frac{1}{n}}\leq\frac{C_1}{2}$.
Furthermore, using \eqref{solblev inequ loc} we derive
\begin{align*}
||\nabla\psi||^2_{2}\leq CF||\nabla\psi||_{2}|A(k)|^{\frac{1}{2}-\frac{1}{p}}.
\end{align*}
Then we apply the H\"older inequality to obtain
\begin{align*}
||\nabla\psi||^2_{2}&\leq \frac{1}{2}||\nabla\psi||^2_{2}+ \frac{1}{2}C^2F^2|A(k)|^{1-\frac{2}{p}}\\
&\leq C^2F^2|A(k)|^{1-\frac{2}{p}}.
\end{align*}
Now using \eqref{solblev inequ loc} again we derive
\begin{align*}
||\psi||_{\frac{2n}{n-1}}\leq CF|A(k)|^{\frac{1}{2}-\frac{1}{p}}
\end{align*}
which provides for any $h>k$,
\begin{align*}
|A(h)|\leq (\frac{CF}{h-k})^{\frac{2n}{n-1}}|A(k)|^{\frac{n(p-2)}{p(n-1)}}.
\end{align*}
Let $d=CF
|A(0)|^{\frac{p-2n}{2np}}2^{\frac{n(p-2)}{p-2n}}$ and $k_s=d(1-\frac{1}{2^s})$.
It is proved inductively that $$|A(k_{s+1})|\leq |A(0)|2^{-\frac{2nps}{p-2n}}.$$ Consequently, $|A(d)|=0$ as $s\rightarrow\infty$. In other words, we obtain
\begin{align*}
\dot\vphi\leq CV^{\frac{p-2n}{2np}}2^{\frac{n(p-2)}{p-2n}}
(||\dot\vphi||_{\frac{2np}{2n+p}}+||S-n||_\infty)
 \text{ in } B_{\frac{1}{2}r}(x).
\end{align*}
Finally, this estimate together with \eqref{dvphi q} provides the global supper bound of $\dot\vphi$ we desired. Similarly, we obtain the global lower bound.
\end{proof}
\begin{rem}
Given the bound of the scalar curvature \eqref{scalar} and the $K$-energy's lower bound, it is available to apply Phong-Sturm's argument \cite{MR2215459} to get
$\int_M|\nabla\dot\vphi|^2\om^n_\vphi\rightarrow0.$
Then we apply the Poincare inequality and the normalization condition \eqref{normalized cond} to obtain
$\int_M|\dot\vphi|^2\om^n_\vphi\rightarrow0.$
\end{rem}
The evolution of
$\mu_0(t)=\frac{1}{V}\int_M\dot\vphi^2\om_\vphi^n$ is given by
\begin{align*}
\frac{\p}{\p t}\mu_0(t)
=-2\int_M(1+\dot\vphi)|\nabla\dot\vphi|^2\om^n_\vphi
+2\int_M\dot\vphi^2\om^n_\vphi.
\end{align*}
Then following the same argument in Chen-Tian \cite{MR1893004},
we use the fact the the spectrum of $g_t$ converges to
the spectrum of the K\"ahler-Einstein metric
and the Futaki invariant is zero
to obtain
\begin{align}\label{0 exponential}
\mu_0(t)=\frac{1}{V}\int_M\dot\vphi^2\om_\vphi^n
\leq\mu_0(0)e^{-\theta t}.
\end{align}
Moreover it is direct to compute that (see for Page 539 in Chen-Tian \cite{MR1893004})
\begin{align}\label{l exponential}
\mu_l(t)=\frac{1}{V}\int_M|\nabla^l\dot\vphi|^2\om_\vphi^n
\leq\mu_l(0)e^{-\theta t}.
\end{align}
We apply the Sobolev imbedding theorem to obtain
\begin{align}\label{potential l exponential}
|\dot\vphi|_{C^l(g_\vphi)}\leq Ce^{-\theta t}.
\end{align}
We use $\vphi(t)=\vphi(0)+\int_0^1\dot\vphi dt$ and \eqref{potential l exponential} to obtain the $C^0$ estimate
\begin{align}
|\vphi(t)|_{C^0}\leq|\vphi(0)|_{C^0}+Ce^{-\theta t}.
\end{align}
From the equation
$\dot\vphi-\log\frac{\om^n_\vphi}{\om^n}
=\vphi+a(t),$
Yau's estimate \cite{MR480350} (see \cite{MR799272}) gives the second order estimate
$0<n+\tri\vphi\leq C$.
It follows $g$ and $g_\vphi$ are equivalent.
So we have $|\dot\vphi-\vphi-a(t)|_{C^2}\leq C$
thanks to the uniform bound of $a(t)$ given by \eqref{a} and \eqref{potential l exponential}.
Then the $C^{2,\a}$ estimates by Evans \cite{MR649348} and Krylov \cite{MR661144}
shows $\vphi$ has uniform $C^{2,\a}$ bound. Moreover, \thmref{regularity of PCF: reg of PCF} provides the uniform bound on all higher order derivatives. Then $g$ and $g_\phi$ are $C^\infty$ equivalent.
So \eqref{l exponential} implies
$$|\vphi-\vphi_\infty|_{C^{k}(g_\infty)}\leq C_ke^{-\theta t},\forall k\geq0.$$
Therefore, we have obtained the exponential convergence of the K\"ahler-Ricci flow.
\begin{prop}\label{sta}
If the K\"ahler-Ricci flow converges to a K\"ahler-Einstein metric in Cheeger-Gromov sense, i.e. for any sequence $g(t_i)$,
there is a subsequence $g(t_{i_j})$
and the holomorphic transformation $\varrho(t_{i_j})$ such that
$\varrho(t_{i_j})^\ast g(t_{i_j})$ converges smoothly to a K\"ahler-Einstein metric $g_\infty$. Then the K\"ahler-Ricci flow must converge exponentially to
a unique K\"ahler-Einstein metric nearby.
\end{prop}
\section{K\"ahler-Ricci soliton}\label{KRS}
In this section we generalize our above argument to the K\"ahler-Ricci solitons. According to Fujiki \cite{MR0481142},
The identity part of holomorphic transformation group $Aut_0(M)$ is meromorphically isomorphic to a linear algebraic group $L(M)$ and such that the quotient $Aut_0(M)/L(M)$ is a complex torus. In Futaki-Mabuchi's work \cite{MR1314584}, they used the Chevalley decomposition to $L(M)$ to obtain
a semidirect decomposition
$$Aut_0(M)=Aut_r(M)\ltimes R_u.$$ Here $Aut_r(M)$ is the reductive algebra group which is the complexification of a maximal compact subgroup $K$ and $R_u$ is the unipotent radical of $Aut_0(M)$.
Let $\eta_r$ be the Lie algebra of $Aut_r(M)$.
A K\"ahler metric $\om$ is called K\"ahler-Ricci soliton,
if there is a holomorphic vector field $X$ such that
\begin{align*}
L_X\om=Ric-\om.
\end{align*}
Tian-Zhu in \cite{MR1768112} proved
the uniqueness of K\"ahler-Ricci soliton for a fixed $X$ in the Lie algebra of $Aut_0(M)$:
\begin{thm}(Tian-Zhu \cite{MR1768112})\label{tz uni}
If $(\om,X)$ and $(\om',X)$ are two K\"ahler-Ricci solitons,
then there are holomorphic transformation groups
$\sigma\in Aut_0(M)$ and $\tau\in Aut_r(M)$ such that
$\sigma^\ast\om=\tau^\ast\sigma^\ast\om'$
and $\sigma^\ast X\in\eta_r$.
\end{thm}
Without loss of generality, we may assume $X\in \eta_r(M)$.
Since $L_{\Im X}\om=0$, $\Im X$ generates a one-parameter isometric group $K_X$. We further choose $K$ such that $K_X\subseteq K$.
According to Proposition 2.1 in Tian-Zhu \cite{MR1915043}, $X$ lies in the center of $\eta_r$.

Since there is a real value function $\theta_X$ such that
$L_X\om=\sqrt{-1}\p\bar\p\theta_X$ with $\int_Me^{\theta_X}\om^n=V$ by the Hodge theory.
Then the potential equation of the K\"ahler-Ricci flow \eqref{KRF potential} is
\begin{equation}
\begin{cases}
\frac{\p\vphi}{\p t}&=\log\frac{\om^n_\vphi}{\om^n}
+\vphi-\theta_X+a(t)\\
\vphi(0)&=\vphi_0.
\end{cases}
\end{equation}
We choose $$a(t)=-\frac{1}{V}\int_M(\log\frac{\om^n_\vphi}{\om^n}
+\vphi-\theta_X)\om^n_\vphi$$ and $I(\vphi_0)=0$ so that the K\"ahler-Ricci flow stays in $\mathcal{H}_0$.
Perelman in \cite{Perelman1} defined a functional called $W$ functional,
\begin{align*}
\mathcal{W}(g,f,\tau)
=(4\pi\tau)^{-\frac{n}{2}}\int_M[\tau(|\nabla f|^2+S)+f-n]e^{-f}dV
\end{align*}
which is invariant under diffeomorphism $\sigma$ and scaling $C$:
$\mathcal{W}(C\sigma^\ast g,\sigma^\ast f,C\tau)=\mathcal{W}(g,f,\tau)$.
And the $\mu$ functional is defined by
\begin{align}
\mu(g,\tau)
=\inf_{(4\pi\tau)^{-\frac{n}{2}}\int_Me^{-f}dV=1}
\mathcal{W}(g,f,\tau)
\end{align}
which is also invariant under diffeomorphism.
The minimum is achieved by some smooth function $f$ satisfing
$\tau[(2\tri f-|\nabla f|^2)+S]+f-n=\mu(g,\tau)$.
The first variation of $\mu(g,\tau)$ at $g'_{ij}=v_{ij}$ for fixed $\tau$ is
\begin{align*}
\mu'(v_{ij},\tau)
&=(4\pi\tau)^{-\frac{n}{2}}
\int_M\{-\tau(v_{ij},Ric+D^2f-\frac{1}{2\tau}g)\}e^{-f}dV_g.
\end{align*}
So the (shrinking) K\"ahler-Ricci soliton is the critical point of $\mu(g,\tau=\frac{1}{2})$.
The gradient flow of the $\mu$ functional equals to \eqref{KRF} with $\l=1$ up to a diffeomorphism generated by $\nabla f$.
So the $\mu$ functional is nondecreasing along the Ricci flow.
Tian-Zhu (see Proposition 2.1 in Tian-Zhu \cite{zhu-2008}) calculated the second variant of this functional
near a K\"ahler-Ricci soliton in the canonical class.
\begin{thm}(Tian-Zhu \cite{zhu-2008})\label{MRF mini}
It holds
\begin{align}\label{2 mu}
\frac{\p^2}{dt^2}\mu(\om+\sqrt{-1}\p\bar\p\vphi)|_{t=0}\leq 0
\end{align} and the equality holds if and only if $\dot\vphi(0)$ is the real part of the holomorphic potential of some holomorphic vector field.
\end{thm}
So the only directions at a K\"ahler-Ricci soliton $\om$ in \eqref{2 mu} vanishes are the the directions tangent to the orbit of $\om$ under the action of $Aut_0(M)$ and
we obtain the following local property of the $\mu$ functional.
\begin{lem}\label{MRF mini}
K\"ahler-Ricci soliton is the local maximum of $\mu(g)$ in the canonical class.
\end{lem}
As a result we deduce that a K\"ahler metric which achieves the maximum value of the $\mu(g)$ functional near a K\"ahler-Ricci soliton must be a K\"ahler-Ricci soliton.

Let $\mathcal{E}_0\subset\mathcal{H}_0$
be the space of potentials of K\"ahler-Einstein solitons
with respect to the holomorphic vector field $X$.
In fact, due to \thmref{tz uni},
$\mathcal{E}_0$ is a single orbit under the action of $Aut_r(M)$.
Moreover, analogously to the extremal metric in Calabi \cite{MR780039}, in the appendix of Tian-Zhu \cite{MR1768112}, their Lemma A.2. and Theorem A shows that the identity component of the holomorphic isometric group of the K\"ahler-Ricci soliton $(\om,X)$ is a maximal compact subgroup of $Aut_r(M)$ containing $K_X$. So $(Aut_r(M),K)$ is a Riemannian symmetric pair and $\mathcal{E}_0$ is $Aut_r(M)$-equivariantly diffeomorphic to the Riemannian symmetric space $Aut_r(M)/K$. Then each geodesic initials from $\om$ in $\mathcal{E}_0$ is written in the form $\g(t)=\exp(t\Re Y)^\ast\om$ for some nonzero $Y$ whose imagine part is a Killing vector field. According to Theorem 3.5 in Mabuchi \cite{MR909015}, we obtain $\g(t)$ is also a geodesic in $\mathcal{K}$.
Thus we obtain:
\begin{lem}$\mathcal{E}_0$ is a finite dimension
totally geodesic submanifold
of $\mathcal{H}_0$.
\end{lem}
If we choose $\om_\rho=\om+\p\bar\p\rho$
such that $\rho$ realizes the shortest distance between
$\psi$ and $\mathcal{E}_0$.
Clearly, $\rho$ is uniquely determined.
In fact, due to \thmref{tz uni}
we obtain a holomorphic diffeomorphism $\sigma\in Aut_r(M)$
such that
$\sigma^\ast\om=
\om_\rho=\om+\sqrt{-1}\p\bar\p\rho\text{ and }\rho\in\mathcal{E}_0$.
Following the analogous argument of Proposition \ref{hol} by using the $\mu$ functional instead of the $K$-energy, we obtain the following proposition.
\begin{prop}\label{hol sol}
Assume $M$ admits a K\"ahler-Ricci soliton $(\om,X)$.
There exits a
small constant $\eps_0$.
If $|\vphi_0|_{C^{2,\a}(M)}\leq\eps_0$,
then there is a unique solution $\vphi(t)$
and the corresponding holomorphic transformation $\varrho(t)\in Aut_r(M)$
such that normalization potential of $\varrho(t)^\ast\om_\vphi(t)$
always stays in $\mathcal{S}$.
Moreover, for any sequence $g(t_i)$,
there is a subsequence $g(t_{i_j})$ such that
$\varrho(t_{i_j})^\ast g_{\vphi(t_{i_j})}$ converges smoothly to $g_\infty$.
\end{prop}
Let $\varsigma$ be generated by $\Re X$ such that $\Re X=(\varsigma^{-1})_\ast\frac{\p}{\p t}\varsigma$,
$\varsigma^\ast\om=\om_\varrho$
and $\phi=\varsigma^\ast\vphi+\varrho$.
We obtain the modified K\"ahler-Ricci flow of the form
\begin{equation}\label{MKRF sol}
\begin{cases}
\frac{\p}{\p t}\om_\phi
&=-Ric(\om_\phi)+\om_\phi+L_{\Re X}\om_\phi,\\
\om_{\vphi(0)}&=\om_{\vphi_0}.
\end{cases}
\end{equation}
The modified potential equation is
\begin{equation}\label{KRF potential sol}
\begin{cases}
\frac{\p\phi}{\p t}&=\log\frac{\om^n_\phi}{\om^n}
+\phi+\Re X(\phi)+a(t),\\
\phi(0)&=\varphi_0,
\end{cases}
\end{equation}
with the normalization condition
\begin{align}\label{sol a}
a(t)=-\frac{1}{V}\int_M(\log\frac{\om^n_\phi}{\om^n}
+\phi+\Re X(\phi))\om^n_\phi.
\end{align}
Here
\begin{align*}
&L_X\om_\phi=Ric-\om+\p\bar\p(X(\phi))
=\p\bar\p \theta_X(\phi)=\p\bar\p (\theta_X+X(\phi)),\\
&L_{\Re X}\om_\phi=\p\bar\p (\theta_X+\Re X(\phi))\text{ and }
L_{\Im X}\om_\phi=\p\bar\p (\Im X(\phi)).
\end{align*}
When the initial datum is $K_X$-invariant, $\varphi$ and $\phi$ are both $K_X$-invariant.
Choose the $K_X$-invariant K\"ahler potential space be $$M_X(\om)=\{\phi\in C^\infty(M)\vert
\om+\p\bar\p\phi>0, \Im X(\phi)=0\}.$$
In Tian-Zhu \cite{MR1915043},
they introduced the modified Futaki invariant
\begin{align*}
F_X(Y)=\int_MY(f-\theta_X)e^{\theta_X}\om^n
\end{align*} for all $X,Y\in\eta(M)$ and the Futaki potential $f$ determined by $\om$ and the modified $K$-energy
\begin{align*}
\mu(\om,\om_{\phi})
&=-\frac{n}{V}\int_0^1\int_M\dot\phi[Ric(\om_\phi)-\om_\phi
-\p\bar\p\theta_X(\phi)\\
&+\p(h_{\om_\phi}-\theta_X(\phi))\wedge\bar\p\theta_X(\phi)]
e^{\theta_X(\phi)}\om_\phi^{n-1}dt.
\end{align*}
for any $\phi$ in $M_X(\om)$.
Along the modified K\"ahler-Ricci flow \eqref{MKRF sol}, the modified Futaki invariant takes the form
\begin{align}\label{sol fu}
F_X(Y)&
=-\int_MY(\dot\phi)e^{\theta_X(\phi)}\om_\phi^n\nonumber\\
&=-\int_M\theta_Y(\tri+X)\dot\phi e^{\theta_X(\phi)}\om_\phi^n.
\end{align}
It is obvious to obtain the evolution of the modified $K$-energy along the modified K\"ahler-Ricci flow, i.e.
\begin{align}\label{sol K t}
\frac{\p}{\p t}\mu(\om,\om_{\phi})
=-\frac{1}{V}\int_M|\nabla\dot\phi|^2e^{\theta_X(\phi)}\om^n_\phi.
\end{align}
Accordingly, the modified $K$-energy is decreasing
along the modified K\"ahler-Ricci flow on $M_X(\om)$.
\begin{thm}(Tian-Zhu \cite{MR1915043})\label{sol lb}
If M admits a K\"ahler-Ricci soliton $(\om,X)$,
the modified Futaki invariant $F_X(Y)=0$
for all $Y\in\eta(M)$,
and the modified $K$-energy has lower bound on $M_X(\om)$.
\end{thm}
By proposition \ref{hol sol}, we have $\om_\psi(t)=\varrho^\ast\om_\vphi(t)$ converges to a K\"ahler-Ricci soliton,
$$\lim_{t\rightarrow\infty}Ric(g_{\psi_{t}})
-\om_{\psi_{t}}-L_X\om_{\psi_{t}}=0.$$
Moreover, since $X$ stays in the center of $\eta_r(M)$, we have the holomorphic transformation
$\varsigma^\ast\varrho^{-1\ast}$
keeps the identity invariant,
\begin{align}\label{sol ricci}
\lim_{t\rightarrow\infty}Ric(g_{\phi_{t}})
-\om_{\phi_{t}}-L_X\om_{\phi_{t}}=0.
\end{align}
After taking trace, we obtain
\begin{align}\label{sol scalar}
\lim_{t\rightarrow\infty}\tri_{\phi_t}\dot\phi
=\lim_{t\rightarrow\infty}S(g_{\phi_t})-n-\tri_{\phi_t}\theta_X(\phi)=0.
\end{align}
Again since $X$ stays in the center of $\eta_r(M)$, we also get
\begin{align*}
\lim_{t\rightarrow\infty}L_X(\sqrt{-1}\p\bar\p\dot\phi)
=\lim_{t\rightarrow\infty}L_X(Ric(g_{\phi_{t}})
-\om_{\phi_{t}}-L_X\om_{\phi_{t}})=0.
\end{align*}
For each time $t$, similarity to \lemref{dot vphi} we apply the De Giorgi iteration
to deduce the $L^\infty$ bound of $\dot\phi$ and $X(\dot\phi)$.
In fact, we obtain
\begin{align}\label{sol dvphi 1,2}
||\dot\phi||_{W^{1,2}}
\leq C||S-n-\tri\theta_X(\phi)||_\infty.
\end{align}Then we derive
\begin{align}\label{sol dot vphi}
||\dot\phi||_\infty\leq C||S-n-\tri\theta_X(\phi)||_\infty
\end{align}
and the $L^\infty$ bound of $X(\dot\phi)-\frac{1}{V}\int_MX(\dot\phi)\om^n_\phi$
\begin{align}\label{sol X dot vphi a}
||X(\dot\phi)-\frac{1}{V}\int_MX(\dot\phi)\om^n_\phi||_\infty\leq C||tr_{g_\phi}L_X(Ric(g_{\phi_{t}})
-\om_{\phi_{t}}-L_X\om_{\phi_{t}})||_\infty.
\end{align}
Moreover, Zhu \cite{MR1817785} provides the following estimate.
 \begin{thm}(Zhu \cite{MR1817785})
 For any $\phi\in M_X(\om)$ there is a constant $C$ depends on $g$ and $X$ such that
\begin{align}\label{sol xdp}
||X(\phi)||_\infty\leq C.
\end{align}
\end{thm}

\begin{lem}The relation between $a(t)$ and the modified $K$-energy is
\begin{align}\label{sol a K}
|a(t)-a(0)|\leq\mu(\om,\om_{\phi_0})-\mu(\om,\om_{\phi_t}).
\end{align}
In addition,
\begin{align}\label{sol a'}
|a'(t)|\leq C\frac{1}{V}
\int_M|\nabla\dot\phi|^2e^{\theta_X(\phi)}\om^n_\phi\leq C||S-n-\tri\theta_X(\phi)||_\infty.
\end{align}
\end{lem}
\begin{proof}
We compute \eqref{sol a} to see that
\begin{align}\label{sol a' e}
a'(t)
&=-\frac{1}{V}\int_M[X(\dot\phi)
-|\nabla\dot\phi|^2])\om^n_\phi.
\end{align} It is obvious that
\begin{align}\label{sol X dot vphi av}
|\frac{1}{V}\int_MX(\dot\phi)\om^n_\phi|\leq C\frac{1}{V}
\int_M|\nabla\dot\phi|^2\om^n_\phi.
 \end{align} Since $\theta_X(\phi)=\theta_X+X(\phi)$, then \eqref{sol a'} follows from \eqref{sol xdp} and \eqref{sol dvphi 1,2}. Hence, integrating with $t$ in both sides of \eqref{sol a' e} and using \eqref{sol K t} we have \eqref{sol a K}.
\end{proof}
So \thmref{sol lb} implies $a(t)$ is uniformly bounded.
Moreover, the $L^\infty$ bound of $X(\dot\phi)$ is obtained by substituting \eqref{sol X dot vphi av} and \eqref{sol dvphi 1,2} in \eqref{sol X dot vphi a},
\begin{align}\label{sol X dot vphi}
||X(\dot\phi)||_\infty\leq C||tr_{g_\phi}L_X(Ric(g_{\phi_{t}})
-\om_{\phi_{t}}-L_X\om_{\phi_{t}})||_\infty+C||S-n-\tri\theta_X(\phi)||_\infty.
\end{align}
It is direct to compute the $t$ derivative of
$v_0(t)=\frac{1}{V}\int_M\dot\phi^2
e^{\theta_X(\phi)-\dot\phi}\om_\phi^n$,
we arrive at
\begin{align*}
\frac{\p}{\p t}v_0(t)
=\frac{1}{V}\int_M[-2(1-\dot\phi)|\nabla\dot\phi|^2
+\dot\phi^2(2-\dot\phi-a')
+2\dot\phi a']e^{\theta_X(\phi)-\dot\phi}\om^n_\phi.
\end{align*}
Hence we obtain for $\delta(t)=||\dot\phi||_\infty$,
\begin{align*}
\frac{\p}{\p t}v_0(t)
&\leq\frac{1}{V}\int_M[-2(1-\delta)e^{-\delta}|\nabla\dot\phi|^2
+\dot\phi^2(2+\delta-a')e^\delta]e^{\theta_X(\phi)}\om^n_\phi\\
&+\frac{1}{V}2a'e^\delta\int_M|\dot\phi|^2 e^{\theta_X(\phi)}\om^n_\phi.
\end{align*}
In which $\delta(t)\rightarrow0$ and $a'(t)\rightarrow0$ follow from \eqref{sol dot vphi} and \eqref{sol a'} respectively.

As it is known in Futaki \cite{MR947341} and Tian-Zhu \cite{MR2221147}, the operator $\tri_g+ X$ for a K\"ahler-Ricci soliton $(g,X)$ is a self-adjoin elliptic operator in the space $C^{\infty}(M,\mathbb{C})$ equipped with the weighted inner product $(f,g)=\int_Mf\bar g e^{\theta_X(g)}\om^n$. Also the first eigenvalue of $\tri_g+X$ is $1$ and moreover the corresponding eigenspace consists the holomorphic potentials of the holomorphic vector fields in $\eta(M)$.
We apply the same method to the K\"ahler-Einstein metric case to use the fact that the spectrum of $g_{\phi}$ converges to
the the spectrum of the K\"ahler-Ricci soliton
and the modified Futaki invariant \eqref{sol fu} vanishes
to obtain
\begin{align}\label{sol 0 exponential}
u_0(t)=\frac{1}{V}\int_M\dot\phi^2\om_\phi^n
\leq u_0(0)e^{-\theta t}.
\end{align}
Moreover we show that
\begin{lem}
\begin{align}\label{sol l exponential}
u_l(t)=\frac{1}{V}\int_M|\nabla^l\dot\phi|^2\om_\phi^n
\leq Cu_l(0)e^{-\theta t}.
\end{align}
\end{lem}
\begin{proof}
We compute
\begin{align*}
\frac{\p u_l(t)}{\p t}&=2\frac{1}{V}\int_M(\nabla_l\dot\phi,\nabla_l\ddot\phi)
+\frac{1}{V}\int_M|\nabla^l\dot\phi|^2\tri\dot\phi\om_\phi^n\\
&-\frac{1}{V}\int_Mg_\phi^{i_1j_1}\cdots g_\phi^{i_{p-1}i_{p+1}} \dot\phi^{i_pj_q}g_\phi^{i_{p+1}j_{q+1}}\cdots g_\phi^{i_lj_l}\nabla_{i_1\cdots i_l}\dot\phi\nabla_{j_1\cdots j_l}\dot\phi\om_\phi^n.
\end{align*}
Since \eqref{KRF potential sol} gives $\ddot\phi=\tri\dot\phi+\dot\phi+X(\dot\phi)$, we use the H\"older inequality to estimate the first term
\begin{align*}
-2u_{l+1}(t)+\eps u_{l+1}(t)+C(\eps)u_{l}(t).
\end{align*}
By \eqref{sol scalar}, the second term is bounded by $u_{l}(t)$.
Using \eqref{KRF potential sol} in the third term we get
\begin{align*}
-\frac{1}{V}\int_Mg_\phi^{i_1j_1}\cdots g_\phi^{i_{p-1}i_{p+1}} (-Ric^{i_pj_q}+g_\phi^{i_pj_q}+X(\phi)^{i_pj_q})g_\phi^{i_{p+1}j_{q+1}} \cdots g_\phi^{i_lj_l}\nabla_{i_1\cdots i_l}\dot\phi\nabla_{j_1\cdots j_l}\dot\phi\om_\phi^n
\end{align*} and by \eqref{sol ricci} we have the third term is controlled by $u_{l}(t)$.
Combining theses three estimates, we then obtain by the $L^p$ interpolation inequality
\begin{align*}
\frac{\p u_l(t)}{\p t}&\leq(-2+\eps)u_{l+1}(t)+(C(\eps)+2)u_{l}(t)\\
&\leq(-2+\eps)u_{l+1}(t)+(C(\eps)+2)(\delta u_{l+1}(t)+C(\delta)u_{0}(t))\\
&\leq(C(\eps)+2)C(\delta)u_{0}(t)
\end{align*} provided $-2+\eps+(C(\eps)+2)\delta<0$.
We thus obtain by \eqref{sol 0 exponential}
\begin{align}\label{pt ul}
\frac{\p u_l(t)}{\p t}\leq Ce^{-\theta t}.
\end{align}
Since $u_l(\psi(t))=\frac{1}{V}\int_M|\nabla^l\dot\psi|^2\om_\psi^n\rightarrow0$ and $u_l(\psi(t))$ is invariant under the holomorphic transformation $\varsigma^\ast\varrho^{-1\ast}$, the lemma follows by integrating \eqref{pt ul} from $t$ to $\infty$.
\end{proof}
Since the Sobolev constant and the Poincare constant are uniformly bounded.
Accordingly, \eqref{sol l exponential} implies by the Sobolev imbedding theorem
\begin{align}\label{sol potential l exponential}
|\dot\phi|_{C^l(g_\phi)}\leq Ce^{-\theta t}.
\end{align}
We apply $\phi(t)=\phi(0)+\int_0^1\dot\phi dt$ to obtain
\begin{align}\label{sol c0}
|\phi(t)|_{C^0}\leq|\phi(0)|_{C^0}+Ce^{-\theta t}.
\end{align}
Consider the equation
$$\log\frac{\om^n_\phi}{\om^n}
=\dot\phi-\phi-X(\phi)-a(t).$$
Yau's computation in \cite{MR480350} and its parabolic adaption in Cao \cite{MR799272} shows:
\begin{align*}
&(\tri_\phi-\frac{\p}{\p t})(e^{-C\phi}(n+\tri\phi))
\geq
e^{-C\phi}\{C(\dot\phi-n)(n+\tri\phi)\\
&+(C+\inf_{i\neq{k}}R_{i\bar{i}k\bar{k}})
(n+\tri\phi)^{\frac{n}{n-1}}e^{\frac{-(\dot\phi-\phi- X(\phi)-a(t))}{n-1}}
+\tri(-\phi- X(\phi))-n^2\inf_{i\neq{k}}R_{i\bar{i}k\bar{k}}\}.
\end{align*}
As shown in Tian-Zhu \cite{MR1768112}, at the maximal point of $e^{-C\vphi}(n+\tri\vphi)$ where we get $\tri \phi_i=C(n+\tri\phi)\phi_i$, it follows
$\tri X(\phi)=X_{ki}\phi_{k\bar i}+CX(\phi)(n+\tri\phi)\leq C(n+\tri\phi)$. As a result, we apply \eqref{sol xdp}, \eqref{sol dot vphi}, \eqref{sol a K} and \eqref{sol c0} to obtain
$$0<n+\tri\phi\leq C,$$
which implies $g$ and $g_\phi$ is $L^\infty$ equivalent.
Then by using Calabi's computation \cite{MR0106487} (see Yau \cite{MR480350}) and the method for dealing with the extra term $\tri X(\phi)$ in Zhu \cite{MR1817785}
we have $C^{3}$ norm of $\phi$ has uniform bound. Moreover, \thmref{regularity of PCF: reg of PCF} implies all higher order derivatives are uniformly bounded and $g$ and $g_\phi$ is $C^\infty$ equivalent.
So \eqref{sol potential l exponential} gives
$$|\phi-\phi_\infty|_{C^{l}(g_\infty)}\leq C_le^{-\theta t}.$$
Finally, we have the exponential convergence of the modified K\"ahler-Ricci flow.
\begin{prop}\label{sol exp}
If the K\"ahler-Ricci flow converges to a K\"ahler-Ricci soliton in Cheeger-Gromov sense. Assume the initial K\"ahler potential is $K_X$-invariant, then the modified K\"ahler-Ricci flow must converge exponentially to
a unique K\"ahler-Ricci soliton nearby.
\end{prop}
We remark here Zhu \cite{zhu-2009} also discussed the stability of K\"ahler-Ricci flow near a K\"ahler-Ricci soliton by using Perelman's estimate \cite{Perelman} and Chen-Tian's energy method \cite{MR1893004}\cite{MR2219236}.
\section{Weak flow}\label{WF}
We can weaken the initial condition
according to Chen-Tian \cite{MR2434691}, Chen-Tian-Zhang \cite{ctz-2008} and Song-Tian \cite{songtian2009}.
Let $a(t)=0$ in \eqref{KRF potential ke},
the potential equation reads
\begin{equation}\label{KRF potential ke no}
\begin{cases}
\frac{\p\vphi}{\p t}&=\log\frac{\om^n_\vphi}{\om^n}+\vphi\\
\vphi(0)&=\vphi_0.
\end{cases}
\end{equation}
We defined $\vphi_0$ is the limit of
$\vphi_s\in PSH(M,\om)\cap L^\infty(M)$ in $L^\infty$ norm.
Meanwhile, $\om_{\vphi_0}\geq0$ in the current sense.
Let the weak solution
be a limit of a sequence of approximate solution by
$$\vphi(t)=\lim_{s\rightarrow0}\vphi(s,t).$$
In their articles, they proved that
\begin{thm}(Chen-Tian \cite{MR2434691}, Chen-Tian-Zhang \cite{ctz-2008}, Song-Tian \cite{songtian2009})\label{weak}
If $\vphi_0$ is defined above with $|\vphi_0|_{L^\infty}\leq A$
and $|\frac{\om_{\vphi_0}^n}{\om^n}|_{L^p(M,\om)}\leq B$ for $p>1$,
there is a unique smooth solution $g_\vphi(t)$ of \eqref{KRF} for $t>0$
such that $$\lim_{t\rightarrow0^+}\vphi(t)=\vphi_0.$$
\end{thm}
The key estimate in their proof is that (see Proposition 3.2 in Song-Tian \cite{songtian2009})
\begin{align}\label{weak bound}
|\vphi(t)|_{C^k}\leq C(t,T,k,A,B) \text{ on } (0,T].
\end{align}
Introduce the space
$$\mathcal{N}(\eps_0;B,p)
=\{\vphi\vert
|\vphi|_{L^\infty}\leq\eps_0,
|\frac{\om_{\vphi}^n}{\om^n}|_{L^p(M,\om)}
\leq B \text{ for some } p>1\}$$ for fixing $B$ and $p$.
Here $B$ and $p$ should be chosen such that $\mathcal{N}(\eps_0;B,p)$ is not a empty set.
Clearly, if $|\vphi_0|_{C^{1,1}}\leq\eps_0$,
then $\vphi_0\in\mathcal{N}(\eps_0,1+(2^n-1)\eps_0,\infty)$.
Actually, we can see that
\begin{lem}\label{weak ini}
When we fix $t_0\in(0,T]$,
for any $\eps_1>0$,
there is a small $\eps_0$,
for any $\vphi_0\in\mathcal{N}(\eps_0;B,p)$,
we have $|\vphi(t_0)|_{C^{2,\a}}\leq\eps_1$.
\end{lem}
\begin{proof}
If the conclusion fails,
we could choose a sequence of $\vphi_s$ such that
$$|\vphi_s|_{L^\infty}\leq\frac{1}{s}
\text{ and }
|\frac{\om_{\vphi_s}^n}{\om^n}|_{L^p(M,\om)}\leq B.$$
But for each corresponding solution $\vphi_s(t)$
constructed by \thmref{weak}, we have
\begin{align}\label{weak ini con}
|\vphi_s(t_0)|_{C^{2,\a}}>\eps_1.
\end{align}
Setting $g_{a\vphi i\bar j}
=\int_0^t(g_{i\bar j}+a\vphi_{i\bar j})da>0$, we rewrite \eqref{KRF potential ke no} as follows
\begin{equation*}
\begin{cases}
\frac{\p\vphi}{\p t}&=\tri_{g_{a\vphi}}\vphi+\vphi\\
\vphi_s(0)&=\vphi_s.
\end{cases}
\end{equation*}
By using the maximal principle we obtain that
\begin{align}\label{weak L infty}
\sup_M|\vphi_s(t_0)|\leq e^{t_0}\sup_M|\vphi_s|.
\end{align}
By \eqref{weak bound},
we can pass a subsequence of $\vphi_{s_i}(t_0)$ such that
$$\lim_{i\rightarrow\infty}\vphi_{s_i}(t_0)=\vphi_\infty(t_0)$$
in $C^k$ for $k\geq0$.
Let $s=s_i$ in \eqref{weak L infty} then the limit
approaches
\begin{align}
\sup_M|\vphi_\infty(t_0)|\leq 0
\end{align}
which contradicts \eqref{weak ini con}.
\end{proof}
Now as we have a $C^{2,\a}$ small initial datum $\vphi(t_0)$,
we normalize it to be $\vphi_0-I(\vphi_0)$
which is also $C^{2,\a}$ small.
Then we can solve equation \eqref{KRF potential ke} with this initial datum.
Therefore combining Proposition \ref{no hol},
Proposition \ref{hol}, Proposition \ref{sta} and \lemref{weak ini} we obtain main \thmref{sta KE}. Analogously, we apply Proposition \ref{hol sol}, Proposition \ref{sol exp} and \lemref{weak ini} to obtain \thmref{sta sol}.
\section{Another choice of holomorphic transformations}\label{ACHT}
In this section, we follows the argument by Bando-Mabuchi
\cite{MR946233} and Chen-Tian \cite{MR1893004}
to find a good holomorphic transformation.
$I$ and $J$ functional are defined as
\begin{align*}
I(\om,\om_\vphi)&=\frac{1}{V}\int_M\vphi(\om^n-\om^n_\vphi),\\
J(\om,\om_\vphi)&=\frac{1}{V}\sum_{i=0}^{n-1}\int_M\frac{i+1}{n+1}
\sqrt{-1}\p\vphi
\wedge\bar\p\vphi\wedge\om^i\wedge\om^{n-1-i}_\vphi.
\end{align*}
From Aubin \cite{MR1636569}
they are both semi-positive functionals and satisfy
\begin{align}\label{IJ}
0\leq I(\om,\om_\vphi)\leq(n+1)(I(\om,\om_\vphi)-J(\om,\om_\vphi))
\leq n I(\om,\om_\vphi).
\end{align} for all $\vphi\in\mathcal{H}$.
Fix $\vphi\in\mathcal{H}_0$, consider a functional
\begin{align*}
\Psi(\sigma)=(I-J)(\om_\vphi,\sigma^\ast\om)=(I-J)(\om_\vphi,\om_\rho)
\end{align*}
for any $\sigma\in Aut_r(M)$
which is the reductive subgroup of $Aut(M)$
and $\sigma^\ast\om=\om+\p\bar\p\rho$.
Since $\om_{\rho}$ is a K\"ahler-Einstein metric, it satisfies
\begin{align}\label{rho}
\log\frac{\om_{\rho}^n}{\om^n}+\rho
=0 \text{ and }
I(\rho)=0.
\end{align}
If $\om_\rho$ is the minimal point of $\Psi$,
for any $u\in\L_1(\om_\rho)$,
we have
\begin{align}\label{orthonomal}
\int_M(\rho-\vphi)u\om^n_\rho=0.
\end{align}
It is known that $\eta(M)\cong\L_1(\om)$ for any K\"ahler-Einstein
metric $\om$ in \cite{MR0094478}.
In order to prove the minimizer of $\Psi$ can always be attained,
it is sufficient to prove
\begin{prop}
For all
$\rho\in\{\rho\vert\sigma^\ast\om=\om_\rho,\sigma\in Aut_r(M),
\Psi(\sigma)\leq r\}$,
we have
$$|\vphi-\rho|_{C^{2,\a}(g_\vphi)}\leq C(|\vphi|_{C^{4,\a}}).$$
\end{prop}
\begin{proof}
Clearly,
$$-\tri_\vphi(\rho-\vphi)<n \text{ and }
 -\tri_\rho(\rho-\vphi)>-n.$$
Since the lower bound of the Green function is given by
\begin{align}\label{low G}
G_\vphi\geq-\gamma\frac{D_\vphi^2}{Vol_\vphi}\doteq-A_\vphi.
\end{align}
Here the volume is a constant in fixed K\"ahler class
and $diam(g_\vphi)\leq Cdiam(g)$ by $|\vphi|_{C^{2}}\leq C$.
Using the Green formula and \eqref{low G}, we obtain
\begin{align}\label{up}
\sup_M(\rho-\vphi)&=\frac{1}{V}\int_M(\rho-\vphi)\om_\vphi^n-\frac{1}
{V}\int_M\tri_\vphi(\rho-\vphi)(y)(G_\vphi(x,y)+A_\vphi)\om_\vphi^n(y)
\nonumber\\
&\leq\frac{1}{V}\int_M(\rho-\vphi)\om_\vphi^n+nA_\vphi.
\end{align}
Similarly, we deduce
\begin{align}\label{low}
\inf_M(\rho-\vphi)&
=\frac{1}{V}\int_M(\rho-\vphi)\om_\rho^n-
\frac{1}{V}\int_M\tri_\rho(\rho-\vphi)(y)(G_\rho(x,y)+A_\rho)
\om_\rho^n(y)\nonumber\\
&\geq\frac{1}{V}\int_M(\rho-\vphi)\om_\rho^n-nA_\rho.
\end{align}
Because
$Ric(\rho)=\om_\rho$, $diam(g_\rho)\leq\sqrt{2n-1}\pi$
by Myers theorem.
Combining \eqref{up} and \eqref{low} we have
\begin{align}\label{osc}
\osc_M(\rho-\vphi)&\geq\frac{1}{V}\int_M(\rho-\vphi)
(\om_\vphi^n-\om_\rho^n)+C(|\vphi|_{C^2}).
\end{align}
From \eqref{IJ} we obtain
$$\frac{1}{V}\int_M(\rho-\vphi)(\om_\vphi^n-\om_\rho^n)
=I(\om_\vphi,\om_\rho)\leq
(n+1)(I-J)(\om_\vphi,\om_\rho)\leq(n+1)r.$$
Since $\om_\rho$ is a K\"ahler-Einstein metric, we have
\begin{align}\label{eq vphi}
(\om_\vphi+\sqrt{-1}\p\bar\p(\rho-\vphi))^n
=e^{-(\rho-\vphi)+h_\vphi}
\om_\vphi^n.
\end{align}
with
\begin{align*}
\sqrt{-1}\p\bar\p h_{\vphi}
=Ric(\om_\vphi)-\om_\vphi
\text { and }
\int_Me^{h_{\vphi}}\om_\vphi^n=Vol(M).
\end{align*}
By using the second order estimate in Yau \cite{MR480350}, we get
\begin{align}
n+\tri_\vphi(\rho-\vphi)&\leq
e^{C\osc_M(\rho-\vphi)}
C(\sup_M(\inf_{i\neq k}|R_{\vphi i \bar i l\bar l}|),
\inf_MS_\vphi,\sup_M h_\vphi)\nonumber\\
&\leq e^{C\osc_M(\rho-\vphi)}C(|\vphi|_{C^4}).
\end{align}
Then the Krylov estimate shows $\rho-\vphi$
has $C^{2,\alpha}$ bound.
\end{proof}
Thus we also obtain the uniform bound of gauge $\rho$.
Our previous discussion implies:
\begin{cor}\label{PCF on SoKM: gauge bound}
If $|\vphi|_{C^{4,\a}}$ is bounded
and $\rho$ is the minimizer of $\Psi$,
then $|\vphi-\rho|_{C^{2,\a}}$ and $|\rho|_{C^{2,\a}}$
are both bounded.
\end{cor}
This lemma implies $g_\rho$ is equivalent to $g$.
We now turn to obtain the uniqueness of the
critical points of the functional $\Psi$
when $\vphi$ is small.
The second variation of $\Psi$ at $\rho$ is given by the formula
\begin{align}\label{2ed va}
D^2\Psi_{\rho}(u,v)
=\frac{1}{V}\int_M(1+\frac{1}{2}\tri_\rho\rho)uv\om^n_\rho.
\end{align}

\begin{lem}\label{uniqueness}
For all $|\vphi|_{C^{2,\a}}\leq \eps_1$ and $u\in\L_1(\om_\rho)$,
the bilinear form $D^2\Psi_{\rho}(u,u)$ is positive.
Hence $\rho$ is unique.
\end{lem}
\begin{proof}
Note that \eqref{eq vphi} can be rewritten as
\begin{align}\label{eq rho}
(\om_\rho+\sqrt{-1}\p\bar\p(\vphi-\rho))^n
=e^{-(\vphi-\rho)-h_\vphi}\om_\rho^n.
\end{align}
By definition, $h_\vphi$ is given by
$$h_\vphi
=-\log\frac{\om_\vphi^n}{\om^n}-\vphi-
\log(\frac{1}{V}\int_Me^{-\vphi})\om^n.$$
We conclude that
\begin{align}\label{h vph}
|h_\vphi|_{C^{2,\a}(g_\rho)}\leq C\eps_1\leq\delta
\end{align}
by assumption of $\vphi$.
Let $$C_\perp^{2,\a}(M)=\{\vphi\in C^{2,\a}(M)\vert
\int_M\vphi u\om_\rho^n, \forall u\in \L_1(\om_\rho)\}.$$
Define the operator of \eqref{eq rho} by
\begin{align*}
\Phi(a,b)
=\log\frac{(\om_\rho+\sqrt{-1}\p\bar\p a)^n}{\om_\rho^n}+a+b,\\
C_\perp^{2,\a}(M)\times C^\a(M)\rightarrow C^\a(M).
\end{align*}
It is clear that $\Phi(\vphi-\rho,h_\vphi)=0$ from \eqref{eq rho}.
The linearized operator of \eqref{eq rho} at $(a,b)=(0,0)$
is given by
$$\delta_a\Phi(v)=\tri_\rho v+v.$$
We infer that $\delta_a\Phi$ is invertible from $C_\perp^{2,\a}(M)$ to $C_\perp^{\a}(M)$.
The implicit function theorem implies
there is a small $\delta$
neighborhood of $0$ in $C^\a(M)$ such that
when $|h_\vphi|_{C^{2,\a}(g_\rho)}\leq\delta$,
we have from \eqref{orthonomal} that
\begin{align}
|\vphi-\rho|_{C^{2,\a}(g_\rho)}\leq C\delta.
\end{align}
Hence we deduce that
\begin{align}
|\rho|_{C^{2,\a}}\leq
|\vphi-\rho|_{C^{2,\a}}+|\vphi|_{C^{2,\a}}\leq C\eps_1<1
\end{align}
by using \corref{PCF on SoKM: gauge bound}, \eqref{h vph}
and choosing appropriate $\eps_1$.
\end{proof}
\bibliography{bib}
\bibliographystyle{plain}

\end{document}